\documentclass[11pt]{article}
\usepackage{geometry}
\usepackage{amsfonts}
\usepackage{amssymb}
\usepackage{graphicx}
\usepackage{amsmath}
\usepackage{amsthm}
\usepackage{enumitem}
\usepackage{tikz}
\usepackage{mathrsfs}
\usepackage{cancel}
\usepackage{todonotes}
\usepackage{cases}
\usetikzlibrary{matrix}
\usetikzlibrary{arrows,shapes}
\usepackage{cite}
\usepackage{bm}
\usepackage{hyperref}
\hypersetup{colorlinks=true, urlcolor=blue}
\expandafter\let\expandafter\oldproof\csname\string\proof\endcsname
\let\oldendproof\endproof
\renewenvironment{proof}[1][\proofname]{%
  \oldproof[\ttfamily \scshape \bf #1. ]%
}{\oldendproof}

\def\O{{\cal O}}
\def\C{{\cal C}}

\def\B{\mathbb{B}}
\def\R{{\rm I\!R}}
\def\oR{\overline{\R}}
\def\N{{\rm I\!N}}

\def\ox{\bar{x}}
\def\oy{\bar{y}}
\def\oz{\bar{z}}
\def\ov{\bar{v}}

\def\ou{\bar{u}}

\def\X{{\bf X}}
\def\Y{{\bf Y}}
\def\Z{{\bf Z}}

\def\S{\bf {S}}

\def \b{{\}_{k\in\N}}}

\def\xk{x^k}

\def\zk{z^k}

\def\tk{{t_k}}

\usepackage{scalerel,stackengine}
\stackMath
\newcommand\nwidehat[1]{%
\savestack{\tmpbox}{\stretchto{%
  \scaleto{%
    \scalerel*[\widthof{\ensuremath{#1}}]{\kern-.6pt\bigwedge\kern-.6pt}%
    {\rule[-\textheight/2]{1ex}{\textheight}}
  }{\textheight}%
}{0.5ex}}%
\stackon[1pt]{#1}{\tmpbox}%
}

\newcommand\nwidecheck[1]{%
\savestack{\tmpbox}{\stretchto{%
  \scaleto{%
    \scalerel*[\widthof{\ensuremath{#1}}]{\kern-.6pt\bigwedge\kern-.6pt}%
    {\rule[-\textheight/2]{1ex}{\textheight}}
  }{\textheight}%
}{0.5ex}}%
\stackon[1pt]{#1}{\scalebox{-1}{\tmpbox}}%
}
\def\what{\nwidehat}
\def\wch{\nwidecheck}

\def\emp{\emptyset}

\def\tto{\rightrightarrows}

\def\prox{{{\rm prox}\,}}

\def\tto{\rightrightarrows}

\def\sub{\partial}
\def\Hat{\widehat}
\def\ra{\rangle}
\def\la{\langle}
\def\ve{\varepsilon}
\def\omu{\bar{\mu}}
\def\lm{\lambda}

 \def\para{{\rm par}\,}
\def\dd{\delta}
\def\al{\alpha}
\def\Th{\Theta}
\def\th{\theta}
\def\vt{\vartheta}

\def\ph{\varphi}

\def\toset_#1{\xrightarrow{#1}}
\DeclareMathOperator*{\mini}{minimize\;}
\DeclareMathOperator*{\argmin}{argmin}
\DeclareMathOperator*{\argmax}{argmax}
\def\d{{\rm d}}
\def\dist{{\rm dist}}

\def\rge{{\rm rge\,}}

\def\ri{{\rm ri}\,}

\def\tr{{\rm tr}\,}
\def\inte{{\rm int}\,}
\def\gph{{\rm gph}\,}
\def\epi{{\rm epi}\,}
\def\dim{{\rm dim}\,}
\def\dom{{\rm dom}\,}

\def\ker{{\rm ker}\,}

\def\aff{{\rm aff}\,}

\def\cl{{\rm cl}\,}
\def\rank{{\rm rank}\,}

\def\sm{\hbox{${1\over 2}$}}

\def\K{{\overline{  K}}}
\def\rN{{\what{N}}}
\def\rt{{\what{T}}}

\def\verl{ \;\rule[-0.4mm]{0.2mm}{0.27cm}\;}

\begin{document}
\begin{center}
{\Large \bf Smoothness of Subgradient Mappings and  Its  Applications in Parametric Optimization}\\[2ex]
NGUYEN T. V. HANG\footnote{School of Physical and Mathematical Sciences, Nanyang Technological University, Singapore 639798 and Institute of Mathematics, Vietnam Academy of Science and Technology, Hanoi, Vietnam (thivanhang.nguyen@ntu.edu.sg or ntvhang@math.ac.vn). Research of this author is partially supported by Singapore National Academy of Science via SASEAF Programme under the grant RIE2025 NRF International Partnership Funding Initiative.} and   EBRAHIM SARABI\footnote{Corresponding author, Department of Mathematics, Miami University, Oxford, OH 45065, USA (sarabim@miamioh.edu). Research of this    author is partially supported by the U.S. National Science Foundation  under the grant DMS 2108546.}
\end{center}
\vspace*{0.05in}

\small{\bf Abstract.} We demonstrate that the concept of strict proto-differentiability of subgradient mappings can play a similar role  as smoothness of the gradient mapping of a  function in the study of subgradient mappings of prox-regular functions.
We then show that metric regularity and strong metric regularity are equivalent for a class of generalized equations when this condition is satisfied.  For a class of composite functions, called ${\cal C}^2$-decomposable, 
we argue that strict proto-differentiability can be characterized via a simple relative interior condition. Leveraging this observation, we present a characterization of the continuous differentiability of the proximal 
mapping for this class of function via a certain relative interior condition. Applications to the study of strong metric regularity of the KKT system of a class of composite optimization problems are also provided. 
\\[1ex]

\newtheorem{Theorem}{Theorem}[section]
\newtheorem{Proposition}[Theorem]{Proposition}

\newtheorem{Lemma}[Theorem]{Lemma}
\newtheorem{Corollary}[Theorem]{Corollary}
\numberwithin{equation}{section}

\theoremstyle{definition}
\newtheorem{Definition}[Theorem]{Definition}
\newtheorem{Example}[Theorem]{Example}
\newtheorem{Remark}[Theorem]{Remark}

\renewcommand{\thefootnote}{\fnsymbol{footnote}}

\normalsize

\section{Introduction}

Smoothness of the gradient mapping of a function, or twice  differentiability of the function itself, has numerous applications including those in designing efficient numerical algorithms for solving optimization problems such as Newton and Newton-like methods.
That inspires us to ask the following question: Is there a counterpart of the latter concept for subgradient mapping of nonsmooth functions?
If so, what should we expect from such a property when  dealing with different second-order variational constructions? 
Borrowing the concept of   {\em strict proto-differentiability}  from  \cite{pr}, we denominate that 
this concept plays a similar role for subgradient mappings as the smoothness of  the gradient mapping. 
Given a finite dimensional Hilbert space $\X$ and  a proper function $f:\X\to \oR:=[-\infty,\infty]$,  recall from  \cite{pr} that the subgradient mapping $\sub f$ is 
said to be strictly proto-differentiable at $\ox$ for $\ov$ with $(\ox,\ov)\in \gph \sub f$ if the regular (Clarke) tangent cone to $\gph \sub f$ at $(\ox,\ov)$, denoted by $\rt_{\gph \sub f}(\ox,\ov)$, and the paratingent cone to $\gph \sub f$ at $(\ox,\ov)$,
denoted by $\widetilde{T}_{\gph \sub f}(\ox,\ov)$, coincide; see \eqref{tan1} for the definitions of the these cones. At first glance, it may not be clear why this concept can be considered as a proper candidate in this regard but 
our developments in this paper will confirm  the unique potential  of this concept  to answer our questions. For instance, it is well-known that the Hessian matrix of a twice continuously differentiable function is symmetric. In Section~\ref{sec3},
we are going to demonstrate that for a broad class of functions, called prox-regular, the graphical derivative and coderivative (generalized Hessian) of  subgradient mappings coincide provided that 
strict proto-differentiability of subgradient mappings is satisfied; see Theorem~\ref{thm:gdcod}. This can be viewed as a far-reaching extension of the aforementioned property of the Hessian matrix of  twice continuously differentiable functions
and also as a motivation to investigate further strict proto-differentiability  for important classes of functions.

Poliquin and Rockafellar obtained interesting characterizations of strict proto-differentiability of subgradient mappings of prox-regular functions in \cite{pr2}. In particular, it was shown there  that 
this concept amounts to the continuous differentiability of proximal mappings under certain assumptions. The latter property of proximal mappings was  studied  for convex sets in Hilbert spaces by Holmes in \cite{hol} in the 1970s
using the implicit function theorem and imposing a certain assumption on the boundary of the convex set under consideration. The characterization in \cite{pr2} indicates that 
concepts of  second-order variational analysis are perhaps  more appropriate to investigate   the continuous differentiability 
of proximal mappings. Indeed, we recently provided in  \cite{HaS23} a simple characterization of the continuous differentiability 
of the proximal mapping of a certain composite function via a relative interior condition by leveraging strict proto-differentiability of their subgradient mappings.

While strict proto-differentiability was introduced in \cite{pr}, it was not studied systematically 
for different classes of functions commonly seen in constrained and composite optimization problems. Recently, the authors characterized it for polyhedral functions, 
  those that their epigraphs are polyhedral convex sets, in \cite{HJS22} and for a class of   composite functions in \cite{HaS22}, obtained from the composition of a polyhedral function and a ${\cal C}^2$-smooth function. 
In these works, it was shown that strict proto-differentiability of subgradient mappings under consideration amounts to   the subgradient 
 taken from the relative interior of the subdifferential set. Such a characterization had two major consequences. First, it allowed the authors to justify the equivalence of metric regularity and strong metric regularity 
for a class of generalized equations at their nondegenerate solutions (see the paragraph after \eqref{deeq} for its definition). And second, we were able to characterize the 
continuous differentiability of the proximal mapping for certain composite functions. 

This paper aims to lay the foundation for the  systematic study of  strict proto-differentiability of  subgradient mappings of prox-regular functions. In doing so, we
utilize a geometric approach by studying first strict smoothness of closed sets, which is different from the path taken in \cite{pr2}. Then, 
we provide a simple characterization of this concept for an important class of functions, called {\em ${\cal C}^2$-decomposable}. This class of functions, as shown by Shapiro in \cite{sh03}, 
encompasses important classes of functions including polyhedral functions and various eigenvalue/singular value functions that often appear in applications. Moreover, we 
present various consequences of strict proto-differentiability in stability properties of generalized equations and use those to characterize the continuous differentiability 
of the proximal mapping. 
  
 The outline of the paper is as follows. We begin in Section~\ref{spdiff} by recalling our notation and characterizing the strict smoothness of sets and   using them to  present 
 characterizations of strict proto-differentiability of  subgradient mappings of prox-regular functions. Section~\ref{sec3} presents some important consequences of 
 strict proto-differentiability for various second-order variational constructions. In Section~\ref{sec04}, we study the relationship between metric regualrity and strong 
 metric regularity of a class of generalized equations under strict proto-differentiability and show that they are, indeed, equivalent. Leveraging then this equivalence, we
 characterize the continuous differentiability of the proximal mapping for prox-regular functions. Section~\ref{chain} is devoted to establishing a chain rule for 
 strict proto-differentiability of ${\cal C}^2$-decomposable functions and deriving a simple characterization of this concept via a relative interior condition. Using this, we study strong metric regularity of the KKT system of a class of composite optimization problems.

\section{Strict Proto-Differentiability of Subgradient Mappings}\label{spdiff}
In what follows,  suppose that $\X$, $\Y$, and $\Z$ are finite dimensional Hilbert spaces.
We denote by $\B$ the closed unit ball in the space in question and by $\B_r(x):=x+r\B$ the closed ball centered at $x$ with radius $r>0$. 
 In the  product space $\X\times \Y$, we use the norm $\|(w,u)\|=\sqrt{\|w\|^2+\|u\|^2}$ for any $(w,u)\in \X\times \Y$.
 Given a nonempty set $C\subset\X$, the symbols $\inte C$,  $\ri C$, $C^*$,  and $\para C$ signify its interior, relative interior, polar cone, and the   linear subspace parallel to the affine hull of $C$, respectively. 
 For any set $C$ in $\X$, its indicator function is defined by $\dd_C(x)=0$ for $x\in C$ and $\dd_C(x)=\infty$ otherwise. We denote
 by $P_C$ the projection mapping onto $C$ and  by $\dist(x,C)$  the distance between $x\in \X$ and a set $C$.
 For a vector $w\in \X$, the subspace $\{tw |\, t\in \R\}$ is denoted by $[w]$. 
 The domain, range, and graph  of a set-valued mapping $F:\X\tto\Y$ are defined, respectively, by  
$ \dom F:= \{x\in\X\big|\;F(x)\ne\emp \}$, $\rge F=\{u\in \Y|\; \exists \,  w\in \X\;\,\mbox{with}\;\; u\in  F(x) \}$, and $\gph F=\{(x,y)\in \X\times \Y|\, y\in F(x)\}$.    
Let $\{C^t\}_{t>0}$ be a parameterized family of sets in $\X$. Its inner and outer limit sets are defined, respectively,  by 
\begin{align*}
\liminf_{t\searrow 0} C^t&= \big\{x\in \X|\; \forall \; t_k \searrow 0 \;\exists \; x^{t_k}\to x \;\;\mbox{with}\;\; x^{t_k}\in C^{t_k}\; \; \mbox{for}\; \;  k\;\; \mbox{sufficiently large}\big\},\\
\limsup_{t\searrow 0} C^t&= \big\{x\in \X|\; \exists \; t_k \searrow 0 \;\exists\;   \; x^{t_k}\to x \;\;\mbox{with}\;\; x^{t_k}\in C^{t_k}\big\};
\end{align*}
see  \cite[Definition~4.1]{rw}. The limit set of $\{C^t\}_{t>0}$ exists if  $\liminf_{t\searrow 0} C^t=\limsup_{t\searrow 0} C^t =:C$,  written as $C^t \to C$ when $t\searrow 0$. 
A sequence $\{f^k\b$ of functions $f^k:\X\to \oR$ is said to {\em epi-converge} to a function $f:\X\to \oR$ if we have $\epi f^k\to \epi f$ as $k\to \infty$, where $\epi f=\{(x,\al)\in \X\times \R|\, f(x)\le \al\}$ is  the epigraph of $f$;
 see \cite[Definition~7.1]{rw} for more details on  epi-convergence. We denote by $f^k\xrightarrow{e} f$ the  epi-convergence of  $\{f^k\b$ to $f$.

Given a nonempty set $\Omega\subset\X$ with $\ox\in \Omega$, the tangent and derivable (adjacent) cones  to $\Omega$ at $\ox$ are defined, respectively,   by
\begin{equation*} 
T_\Omega(\ox) = \limsup_{t\searrow 0} \frac{\Omega - \ox}{t} \quad \mbox{and}\quad \wch{T}_\Omega(\ox) = \liminf_{t\searrow 0} \frac{\Omega - \ox}{t}. 
\end{equation*}
Consider a set-valued mapping $F:\X\tto \Y$. According to \cite[Definition~8.33]{rw}, the {\em graphical derivative} of $F$ at $\ox$ for $\oy$ with $(\ox,\oy) \in \gph F$ is the set-valued mapping $DF(\ox, \oy): \X\tto \Y$ defined via the tangent cone to $\gph F$ at $(\ox, \oy)$ by
\begin{equation*}
\eta \in DF(\ox, \oy)(w)\; \Longleftrightarrow\; (w,\eta)\in T_{\gph F}(\ox, \oy),
\end{equation*}
or, equivalently, $\gph DF(\ox, \oy) = T_{\gph F}(\ox, \oy)$.  
When $F(\ox)$ is a singleton consisting of $\oy$ only, the notation $DF(\ox, \oy)$ is simplified to $DF(\ox)$. It is easy to see that 
for a single-valued mapping $F$ which is differentiable at $\ox$,  the graphical derivative  $DF(\ox)$ boils down to the Jacobian of $F$ at $\ox$, denoted by $\nabla F(\ox)$. 
The set-valued mapping $F$ is said to be {\em proto-differentiable} at $\ox$ for $\oy$ if $\wch T_{\gph F}(\ox, \oy)=T_{\gph F}(\ox, \oy)$.   
Note that the inclusion `$\subset$' in the later equality always holds, and so proto-differentiability requires that opposite inclusion   be satisfied. 
Proto-differentiability was introduced in \cite{r89} and has been studied extensively for subgradient mappings of different classes of functions in \cite{mms,ms20,pr, rw}. 

A more restrictive version of proto-differentiability was introduced in \cite{pr} and  is the main subject of our study in this paper. To present its definition, recall that 
the regular (Clarke) tangent cone and the paratingent cone to $\Omega$ at $\ox$   are defined, respectively,   by
\begin{equation}\label{tan1}
 \rt_{\Omega}(\ox) =\liminf_{x \xrightarrow{\Omega}\ox,t\searrow 0} \frac{\Omega-x}{t}\quad \mbox{and}\quad  \widetilde{T}_{\Omega}(\ox) =\limsup_{x \xrightarrow{\Omega}\ox,t\searrow 0} \frac{\Omega-x}{t},
\end{equation}
where the symbol $x \xrightarrow{\Omega}\ox$ means that $x\to \ox$ with $x\in \Omega$. 
It can be immediately  observed that the inclusions 
\begin{equation}\label{inct}
\rt_{\Omega}(\ox) \subset  \wch{T}_\Omega(\ox)\subset   {T}_\Omega(\ox)\subset \widetilde{T}_{\Omega}(\ox)
\end{equation}
always hold.  
Below, we record two important properties of the paratingent cone, used later in our developments. 
\begin{Proposition} \label{parat} Assume that $\Omega$ is a subset of $\X$ and $\ox\in \Omega$. Then the following properties hold for the paratingent cone to $\Omega$ at $\ox$.
\begin{itemize}[noitemsep,topsep=2pt]
\item [ \rm {(a)}]  $ \widetilde{T}_{\Omega}(\ox)=- \widetilde{T}_{\Omega}(\ox)$.
\item [ \rm {(b)}] $ \limsup_{x\xrightarrow{\Omega}\ox} T_\Omega(x)\subset \widetilde{T}_{\Omega}(\ox)$.
\end{itemize}
\end{Proposition}
\begin{proof} While the property in (a) is well-known, we provide a short proof.
Take $w\in  \widetilde{T}_{\Omega}(\ox)$. By definition, we find $t_k\searrow 0$, $x^k\xrightarrow{\Omega}\ox$, and $w^k\to w$ such that $\tilde x^k:=x^k+t_kw^k\in \Omega$ for all $k$ sufficiently large.
Thus, we get $\tilde x^k+t_k(-w^k)=x^k\in \Omega$. Since $-w^k\to -w$, we arrive at $-w\in  \widetilde{T}_{\Omega}(\ox)$, which proves (a). The property in (b) results from a standard diagonalization argument and 
can be gleaned from \cite[Proposition~4.5.6]{af}. 
\end{proof}

We proceed with recalling the definitions of  smoothness of a set and a Lipschitzian manifold, appeared first in  \cite[pp.~169--173]{r85}.
 
\begin{Definition} \label{sss} Assume that $\Omega$ is a subset of $\X$ and $\ox\in \Omega$.  
\begin{itemize}[noitemsep,topsep=2pt]
\item [ \rm {(a)}] The set $\Omega$ is called {  smooth} at $\ox$ if $\wch T_{\Omega}(\ox) ={T}_{\Omega}(\ox)$ and $T_{\Omega}(\ox)$ is a linear subspace of $\X$.
\item [ \rm {(b)}] The set $\Omega$ is called { strictly smooth} at $\ox$ if $\rt_{\Omega}(\ox) =\widetilde{T}_{\Omega}(\ox)$. 
\item [ \rm {(c)}] The set $\Omega$ is called a Lipschitzian manifold around $\ox$ if there are an open neighborhood $O$ of $\ox$, a splitting $\X=\Y\times \Z$ with $\Y$ and $\Z$ being two finite dimensional Hilbert spaces,
and  a ${\cal C}^1$-diffeomorphism $\Phi$ from $O$ onto an open  subset $U:=U'\times U''$ of  $\Y\times \Z$ such that $\Phi(\Omega\cap O)=(\gph f)\cap U$, where 
$f:U'\to U''$ is a Lipschitz continuous function. 
\end{itemize}
\end{Definition}

We are now going to present characterizations of strict smoothness of sets that are Lipschitzian manifolds. 
To do so, recall that the  regular normal cone to $\Omega\subset \X$ at $\ox\in \Omega$,  denoted by $\rN_\Omega(\ox)$,     is defined by
 $\rN_\Omega(\ox) = T_\Omega(\ox)^*$. For $\ox\notin \Omega$, we set $\rN_\Omega(\ox) = \emptyset$. The (limiting/Mordukhovich) normal cone $N_\Omega(\ox)$ to $\Omega$ at $\ox$ is
 the set of all vectors $\ov\in \X$ for which there exist sequences  $\{x^k\b\subset \Omega$ and  $\{v^k\b$ with $v^k\in \rN_\Omega( \xk)$ such that 
$(x^k,v^k)\to (\ox,\ov)$. When $\Omega$ is convex, both normal cones boil down to that of convex analysis.  
Recall also that for a  set-valued mapping $F:\X\tto \Y$ with $(\ox,\oy)\in \gph F$,  the coderivative mapping of   $F$ at $\ox$ for $\oy$, denoted  $D^*F (\ox,\oy )$,   is defined by 
\begin{equation}\label{coder}
\eta\in D^*F (\ox,\oy )(w)\iff (\eta,-w)\in N_{\gph F}(\ox, \oy).
\end{equation}
\begin{Theorem} [characterizations of strictly smooth sets] \label{strict}Assume that $\Omega$ is a subset of $\X$, $\ox\in \Omega$, and that $\Omega$ is locally closed around $\ox$. Consider the following properties:
\begin{itemize}[noitemsep,topsep=2pt]
\item [ \rm {(a)}]   $\Omega$ is   strictly smooth at $\ox$;
\item [ \rm {(b)}] $ \lim_{x\xrightarrow{\Omega}\ox} T_\Omega(x)=  {T}_{\Omega}(\ox)$.
\end{itemize}
Then, the implication {\rm(}a{\rm)}$\implies${\rm(}b{\rm)} always holds. If, in addition, $\Omega$ is a Lipschitzian manifold around $\ox$, then the opposite implication holds as well. In this case,  {\rm(}a{\rm)} and {\rm(}b{\rm)} are also equivalent
to the following properties:
\begin{itemize}[noitemsep,topsep=2pt]
\item [ \rm {(c)}]  $T_\Omega(x)$ converges to   ${T}_{\Omega}(\ox)$  as $x\to \ox$  in the set of  points $x\in \Omega$ for which  $\wch T_{\Omega}(x) ={T}_{\Omega}(x)$.
\item [ \rm {(d)}]  $T_\Omega(x)$ converges to   ${T}_{\Omega}(\ox)$  as $x\to \ox$  in the set of  points $x\in \Omega$ at which $\Omega$ is smooth. 
\item [ \rm {(e)}] $T_\Omega(x)$ converges   as $x\to \ox$  in the set of  points $x\in \Omega$ at which $\Omega$ is smooth. 
\end{itemize}
\end{Theorem}

\begin{proof}  If (a) holds,  it follows from \cite[Theorem~6.26]{rw} that
\begin{equation*}\label{lif}
\liminf_{x \xrightarrow{\Omega}\ox} T_{\Omega}(x)= \widehat T_{\Omega}(\ox).
\end{equation*}
Combining this and  Proposition~\ref{parat}(b) tells us that 
$$
\widehat T_{\Omega}(\ox)=\liminf_{x \xrightarrow{\Omega}\ox} T_{\Omega}(x)\subset  \limsup_{x\xrightarrow{\Omega}\ox} T_\Omega(x)\subset \widetilde{T}_{\Omega}(\ox),
$$ 
which, together with (a), yields (b).  The implications (b)$\implies$(c)$\implies$(d)$\implies$(e) clearly hold.
Assume now that $\Omega$ is a Lipschitzian manifold around $\ox$ and turn to the implication  (e)$\implies$(a).
Using the notation in Definition~\ref{sss}(c), we conclude that the the convergence in (e) amounts to the convergence of $\gph \nabla f(u)$ as $u\to \ou$ in the set of points $u$ at which $f$ is differentiable, where $\ou\in U'$ with $(\ou,f(\ou))=\Phi(\ox)$.
Since $f$ is Lipschitz continuous around $\ou$, $\nabla f(u)$ is locally uniformly bounded around $\ou$, whenever it exists.
It follows from \cite[Theorem~5.40]{rw} that the aforementioned graphical convergence of $\nabla f(u)$ coincides with its pointwise convergence as $u\to\ou$ in the same set.
This, along with \cite[Theorem~9.62]{rw}, tells us that $D^*f(\ou)$ is single-valued. Appealing now to  \cite[Exercise 9.25(c)]{rw}
implies that $f$ is strictly differentiable at $\ou$. By \cite[Proposition~3.1(b)]{r85}, the latter is equivalent to $\gph f$ being strictly smooth at $(\ou,f(\ou))$, which amounts to $\Omega$ being strictly smooth at $\ox$,  since $\Phi$ in Definition~\ref{sss}(c) is a diffeomorphism. 
This completes the proof.
\end{proof}

Recall from \cite[Definition~9.53]{rw} that the strict graphical derivative 
of  a set-valued mapping $F$ at $\ox$ for $\oy$ with $(\ox,\oy)\in \gph F$, is  the set-valued mapping $\widetilde D F(\ox,\oy):\X\tto \Y$, defined  by 
\begin{equation*}
\eta \in \widetilde DF(\ox, \oy)(w)\; \Longleftrightarrow\; (w,\eta)\in \widetilde T_{\gph F}(\ox, \oy).
\end{equation*} 
Following \cite{pr}, we say that the set-valued mapping $F$ is  {\em strictly} proto-differentiable at $\ox$ for $\oy$ provided that   $\gph F$ is strictly smooth at $(\ox,\oy)$, namely  $\rt_{\gph F}(\ox, \oy)=\widetilde T_{\gph F}(\ox, \oy)$. 
Indeed, since  the inclusion `$\subset$'  always holds due to  \eqref{inct}, the strict proto-differentiability of $F$ at $\ox$ for $\oy$ amounts to the validity of  the opposite inclusion therein. 
According to \eqref{inct}, the strict proto-differentiability of $F$ at $\ox$ for $\oy$ implies that $\widetilde T_{\gph F}(\ox, \oy)= T_{\gph F}(\ox, \oy)$, which in turn demonstrates that 
$  DF(\ox, \oy)(w)=\widetilde DF(\ox, \oy)(w)$ for any $w\in \X$. While proto-differentiability has been studied for different classes of functions, its strict version has not received much attention.
In fact, we recently characterized in \cite[Theorem~3.2(c)]{HJS22} this property for subgradient mappings of polyhedral functions via a relative interior condition; see also \cite[Theorem~3.10]{HaS22} for a similar 
result for  certain composite functions. Our main objective is to extend the latter characterization for a rather large class of composite functions.

Given a function $f:\X \to \oR$ and a point $\ox\in\X$ with $f(\ox)$ finite, 
a vector $v\in \X$ is called a subgradient of $f$ at $\ox$ if $(v,-1)\in N_{\epi f}(\ox,f(\ox))$. The set of all subgradients of $f$ at $\ox$
is    denoted by $\sub f(\ox)$.   
Replacing the limiting normal cone with $\rN_{\epi f}(\ox, f(\ox))$ in the definition of $\sub f(\ox)$ gives us $\widehat \sub f(\ox)$, which is called the regular subdifferential of $f$ at $\ox$.
A function  $f\colon\X\to\oR$ is called prox-regular at $\ox$ for $\ov$ if $f$ is finite at $\ox$ and locally lower semicontinuous (lsc)   around $\ox$ with $\ov\in\sub f(\ox)$, and there exist 
constants $\ve>0$ and $r\ge 0$ such that
\begin{equation}\label{prox}
\begin{cases}
f(x')\ge f(x)+\la v,x'-x\ra-\frac{r}{2}\|x'-x\|^2\;\mbox{ for all }\; x'\in\B_{\ve}(\ox)\\
\mbox{whenever }\;(x,v)\in(\gph\sub f)\cap\B_{\ve}(\ox,\ov)\; \mbox{ with }\; f(x)<f(\ox) +\ve.
\end{cases}
\end{equation}
The function $f$ is called subdifferentially continuous at $\ox$ for $\ov$ if the convergence $(x^k,v^k)\to(\ox,\ov)$ with $v^k\in\sub f(x^k)$ yields $f(x^k)\to f(\ox)$ as $k\to\infty$. 
Important examples of prox-regular and subdifferentially continuous functions are convex functions and strongly amenable functions in the sense of \cite[Definition~10.23]{rw}. 
Below, we provide a characterization of strict proto-differentiability of subgradient mappings of prox-regular functions, which is a slight improvement of \cite[Corollary~4.3]{pr2}. 

In what follows, we say that  a sequence of  set-valued mappings $F^k:\X\tto \Y$, $k\in \N$,     {\em graph}-converges to $F:\X\tto\Y$ if the sequence of sets $\{\gph F^k\b$ is convergent to $\gph F$. 

\begin{Corollary}[characterizations of strict proto-differentiability] \label{sted}
Assume that  $f:\X\to \oR$ is prox-regular and subdifferentially continuous at $\ox$ for $\ov\in \sub f(\ox)$.
Then the following properties are equivalent:
\begin{itemize}[noitemsep,topsep=2pt]
\item [ \rm {(a)}] $\sub f$ is strictly proto-differentiable at $\ox$ for $\ov$;
\item [ \rm {(b)}] $D(\sub  f)(x,v)$ graph-converges to $D(\sub  f)(\ox,\ov)$  as $(x,v)\to (\ox,\ov)$  and  $(x,v)\in\gph \sub f$;
\item [ \rm {(c)}] $D(\sub  f)(x,v)$ graph-converges to $D(\sub  f)(\ox,\ov)$  as $(x,v)\to (\ox,\ov)$  in the set of pairs  $(x,v)\in\gph \sub f$ for which $\sub f$ is   proto-differentiable;
\item [ \rm {(d)}] $D(\sub  f)(x,v)$ graph-converges to $D(\sub  f)(\ox,\ov)$  as $(x,v)\to (\ox,\ov)$  in the set of pairs  $(x,v)\in\gph \sub f$ for which $\sub f$ is   proto-differentiable and $T_{\gph \sub f}(x,v)$ is a linear subspace. 
\end{itemize}
\end{Corollary}

\begin{proof} Recall that (a) amounts to   $\gph\partial f$ being strictly smooth at $(\ox, \ov)$  and that for all $(x, v) \in \gph \partial f$ we have $\gph D(\partial f)(x, v) = T_{\gph\partial f}(x, v)$.
Also, as mentioned above, we know that proto-differentiability of $\partial f$ at $x$ for $v$, for $(x, v)\in \gph \partial f$, amounts to $\wch T_{\gph \partial f}(x, v)=T_{\gph \partial f}(x, v)$.
According to \cite[Theorem~4.7]{pr},  the graphical set  $\gph\sub f$ is a Lipschitzian manifold around  $(\ox,\ov)$. 
The equivalence of (a)-(d) is a direct consequence of Theorem~\ref{strict}.
\end{proof}

Note that the equivalence of (a), (c), and (d) in Corollary~\ref{sted} were observed before in \cite[Corollary~4.3]{pr2} using an argument via the Moreau envelope of prox-regular functions.
While  the properties in (c) and (d) also appeared in \cite[Corollary~4.3]{pr2}, the latter didn't determine to which mapping   the graph-convergences happen. Also, the property in (b)
is new and was not observed in \cite{pr2}.

\section{Remarkable Consequences of Strict Proto-Differentiability}\label{sec3}

This section is devoted to establishing some interesting consequences of strict proto-differentiability of subgradient mappings of prox-regular functions, which 
highlight the importance of this property in second-order variational analysis of various classes of optimization problems. We begin with recalling the definition of generalized quadratic form from \cite[Definition~2.1]{nt}. 

\begin{Definition} Suppose that  $\ph:\X\to \oR$ is a proper function. 
\begin{itemize}[noitemsep,topsep=2pt]
\item [ \rm {(a)}] We say that the subgradient mapping $\sub \ph:\X\tto \X$ is {  generalized linear} if its graph is a linear subspace of $\X\times \X$.
\item [ \rm {(b)}]  We say that $\ph$ is a generalized quadratic form on $\X$ if $\dom \ph$ is a linear subspace of $\X$ and there exists a linear symmetric mapping $A$ from $\dom \ph$ to $\X$ {\rm(}i.e. $\la A(x),y\ra=\la x,A(y)\ra$
for any $x,y\in \dom \ph${\rm)} such that $\ph$ has a representation of form 
$$
\ph(x)=\la A(x),x\ra \quad \mbox{for all}\;\; x\in \dom \ph.
$$ 
\end{itemize}
 \end{Definition}
It was shown in  \cite[Theorem~2.5]{nt} that 
 for a proper lsc convex function $\ph:\X\to \oR$ with $\ph(0)=0$,  $\ph$ is a generalized quadratic form on $\X$ if and only if $  \sub \ph$ is generalized linear. 
In this case,  one can glean from the proof of the latter result that 
 \begin{equation}\label{gqf}
 \ph(x)=  \sm \la A(x),x\ra+\dd_S(x) \quad \mbox{for}\;\; x\in \X,
\end{equation}
 where $S:=\dom \sub \ph$ is a linear subspace of $\X$ and $A: S\to S$ is linear and symmetric mapping defined by  $A(x)= P_{S}(y)$ for an arbitrary $y\in \partial \varphi(x)$\footnote{Interested readers can find more details about this representation in the arXiv version of this paper available  at \url{ arxiv.org/abs/2311.06026}}.
 Define now $\widehat A:\X\to \X$ by  $\widehat A(x)=A (P_S(x))$ for $x\in \X$. 
It is not hard to  see that $\widehat A$ is a linear symmetric mapping on $X$ satisfying $\widehat A (x) = A(x)$ for all $x\in S$, and therefore can equivalently represent $\ph$ as follows
\begin{equation*}
 \ph(x)=  \sm \la \widehat A(x),x\ra+\dd_S(x) \quad \mbox{for all}\;\; x\in \X.
\end{equation*}

To present our first result in this section, we begin by recalling the concept of second subderivatives. Given a function $f:\X\to \oR$ and $\ox\in \X$ with $f(\ox)$ finite, 
  the {second subderivative} of $f$ at $\ox$ for $\ov\in \sub f(\ox)$ is   defined  by 
\begin{equation*}\label{ssd}
\d^2 f(\bar x , \ov)(w)= \liminf_{\substack{
   t\searrow 0 \\
  w'\to w
  }} \Delta_t^2 f(\ox , \ov)(w'),\;\; w\in \X,
\end{equation*}
where $\Delta_t^2 f(\ox , \ov)(w'):=(f(\ox+tw')-f(\ox)-t\langle \ov,\,w'\rangle)/\frac {1}{2}t^2$  is second-order difference quotients of $f$ at $\ox$ for $\ov$. 
According to \cite[Definition~13.6]{rw},   $f$ is called  {\em twice epi-differentiable} at $\bar x$ for $\ov$ if  the functions $ \Delta_t^2 f(\bar x , \ov)$ epi-converge to $  \d^2 f(\bar x,\ov)$ as $t\searrow 0$. 
Twice epi-differentiability of a prox-regular function is equivalent to proto-differentiability of its subgradient mapping, according to \cite[Theorem~13.40]{rw}, and has been studied extensively in \cite{mms,ms20,ms23,rw}
for different classes of functions. Following \cite[Definition~8.45]{rw}, we say that  $v \in \X$ is  a proximal subgradient of $f$ at $\ox$  if there exists $r \ge 0$ and   a neighborhood $U$ of $\ox$ such that for all $x \in U$, one has
\begin{equation*}\label{proxs}
f(x) \geq f(\ox) + \la v , x - \ox \ra - \frac{r}{2} \|x - \ox \|^2. 
\end{equation*}
Recall also that a set $\Omega\subset \X$ is called regular at $\ox\in \Omega$ if it is locally closed around $\ox$ and $\rN_\Omega(\ox)=N_\Omega(\ox)$. Below, we show that 
strict proto-differentiability of subgradient mappings immediately implies that these subgradient mappings are generalized linear. 
While parts (a) and (b) of the following result were observed before in \cite[Corollary~4.3]{pr2}, we provide a different proof for  (a). The proof of (b) should be known but we could find 
it in any publications. We supply a proof of (b) as well  for the readers' convenience.

\begin{Proposition} \label{glp} Assume that $f:\X\to \oR$ is prox-regular and subdifferentially continuous at $\ox$ for  $\ov\in   \sub f(\ox)$ and that $\sub f$ is strictly proto-differentiable at $\ox$ for $\ov$. Then 
the following properties hold.
\begin{itemize}[noitemsep,topsep=2pt]
\item [ \rm {(a)}] $D( \sub f)(\ox,\ov)$ is generalized linear.
\item [ \rm {(b)}]  The set $\K:=\dom \d^2 f(\ox,\ov)$ is a linear subspace and there is a linear symmetric mapping $A:\X\to \X$ for which the second subderivative of $f$  
at $\ox$ for $\ov$ has a representation of the form  
\begin{equation}\label{gqf1}
\d^2 f(\ox,\ov)(w)= \sm \la A(w),w\ra+\dd_{\K}(w), \quad w\in \X.
\end{equation}
Consequently, $\d^2 f(\ox,\ov)$ is a generalized quadratic form on $\X$. 
\item [ \rm {(c)}]  $\gph \sub f$ is regular at $(\ox,\ov)$.
 \end{itemize}
\end{Proposition}

\begin{proof} Strict proto-differentiability of $\sub f$ at $\ox$ for $\ov$ amounts to saying that $\rt_{\gph \sub f}(\ox,\ov)=\widetilde T_{\gph \sub f}(\ox,\ov)=T_{\gph \sub f}(\ox,\ov)$.
By Proposition~\ref{parat}(a), we have $\widetilde T_{\gph \sub f}(\ox,\ov)=-\widetilde T_{\gph \sub f}(\ox,\ov)$. Since  the regular tangent cone $\rt_{\gph \sub f}(\ox,\ov)$ is a convex cone (cf. \cite[Theorem~6.26]{rw}),
we conclude that  $T_{\gph \sub f}(\ox,\ov)$ is a linear subspace of $\X\times \X$. This, together  with   $\gph D( \sub f)(\ox,\ov)=T_{\gph \sub f}(\ox,\ov)$, confirms that $D( \sub f)(\ox,\ov)$ is generalized linear and hence proves (a).

To justify (b), observe first that strict proto-differentiability of $\sub f$ at $\ox$ for $\ov$ implies proto-differentiability of $\sub f$ at $\ox$ for $\ov$. So, by 
 \cite[Theorem~13.40]{rw}, we have  $  D( \sub f)(\ox,\ov)=\sub (\sm \d^2 f(\ox,\ov))$. Prox-regularity of $f$ at $\ox$ for $\ov$ implies that $\ov$ is a proximal subgradient of $f$ at $\ox$. This, together with \cite[Proposition~2.1]{ms20}, tells us that
$\d^2 f(\ox,\ov)$ is a proper function. Because the second subderivative $\d^2 f(\ox,\ov)$ is positively homogeneous of degree $2$, we arrive at $\d^2 f(\ox,\ov)(0)=0$. By \cite[Proposition~13.49]{rw}, there exists $\rho\ge 0$ such that the function $\ph$, defined by 
$\ph(w)=\d^2 f(\ox,\ov)(w)+ \rho\|w\|^2$ for any $w\in \X$, is convex. Moreover,  $\ph(0)=0$, $\ph$ is lsc, and $\sub \ph(w)=2 D( \sub f)(\ox,\ov)(w)+2 \rho w$. By (a), $D( \sub f)(\ox,\ov)$ is generalized linear, and  so is $\sub \ph$. 
Particularly, $\K =\dom \partial \ph$ is a linear subspace. According to the discussion after \eqref{gqf}, there is a linear symmetric mapping $\widehat A:\X\to \X$ 
such that $\ph(w)=\sm \la \widehat A(w), w\ra+\dd_{\K}(w)$ for any $w\in \X$. Setting $A:=\widehat A-2\rho I$ with $I$ being the identity mapping from $\X$ onto $\X$, we arrive at  \eqref{gqf1},
which proves (b). 

Turing now to (c), we deduce from strict proto-differentiability of $\sub f$ at $\ox$ that $\rt_{\gph \sub f}(\ox,\ov) =T_{\gph \sub f}(\ox,\ov)$.
Employing \cite[Corollary~6.29(b)]{rw} illustrates that $\gph \sub f$ is regular at $(\ox,\ov)$, and hence completes the proof.
\end{proof}

Given a function $f:\X \to \oR$ and  $\ox\in\X$ with $f(\ox)$ finite, the subderivative function $\d f(\ox)\colon\X\to\oR$ is defined by
\begin{equation*}\label{fsud}
\d f(\ox)(w)=\liminf_{\substack{
   t\searrow 0 \\
  w'\to w
  }} {\frac{f(\ox+tw')-f(\ox)}{t}}.
\end{equation*}
The critical cone of $f$ at $\ox$ for   $\bar v\in   \sub f(\ox)$ is defined by 
\begin{equation*}\label{cricone}
{K_f}(\ox,\bar v)=\big\{w\in \X\,\big|\,\la\bar v,w\ra=\d f(\ox)(w)\big\}.
\end{equation*}
For $f=\dd_\Omega$, the indicator function of a nonempty subset $\Omega\subset \X$, the critical cone of $\dd_\Omega$ at $\ox$ for $\ov$ is denoted by $K_\Omega(\ox,\ov)$. In this case, the above definition of the critical cone of a function 
boils down to  the well known concept of a critical cone of a set (see \cite[page~109]{DoR14}), namely $K_\Omega(\ox,\ov)=T_\Omega(\ox)\cap [\ov]^\perp$.
Recall also from \cite[Definition~7.25]{rw} that an lsc function $f: \X\to \oR$ is said to be subdifferentially regular at $\ox\in \X$  if    $f(\ox)$ is   finite and $\rN_{\epi f}(\ox, f(\ox))= N_{\epi f}(\ox,f(\ox))$. 
Below, we record a simple observation about the critical cone of prox-regular function, used often in this paper.

\begin{Proposition} \label{proxcri} Assume that $f:\X\to \oR$ is prox-regular and subdifferentially continuous at $\ox$ for  $\ov\in   \sub f(\ox)$.
Then the following properties are satisfied. 
\begin{itemize}[noitemsep,topsep=2pt]
\item [ \rm {(a)}]  The second subderivative $\d^2 f(\ox,\ov)$ is a proper function and $$\cl\big(\dom \d^2 f(\ox,\ov)\big)\subset K_f(\ox,\bar v)\subset N_{\Hat\sub f(\ox)}(\ov).$$
\item [ \rm {(b)}] If, in addition,  $f$ is subdifferentially regular at $\ox$, then we have   ${K_f}(\ox,\bar v)= N_{ \sub f(\ox)}(\ov)$ and consequently  $K_f(\ox,\ov)$ is a closed convex cone. 
\end{itemize}
\end{Proposition}

\begin{proof} 
Observe first that prox-regularity of $f$ at $\ox$ for $\ov$ implies that $\ov$ is a proximal subgradient of $f$ at $\ox$. This, together with \cite[Proposition~2.1]{ms20}, implies that
$\d^2 f(\ox,\ov)$ is a proper function. Appealing now to \cite[Proposition~13.5]{rw} gives us the inclusions  $\dom \d^2 f(\ox,\ov)\subset K_f(\ox,\bar v)\subset N_{\Hat\sub f(\ox)}(\ov)$. We 
are now going to show that $K_f(\ox,\ov)$ is closed. 
To this end, recall that $\ov$ is a proximal subgradient of $f$ at $\ox$, which results in $\ov\in \Hat\sub f(\ox)$. 
Employing \cite[Exercise~8.4]{rw} then allows us to represent $K_f(\ox, \ov)$ as the  level set $\big\{w\in \X\,\big|\,\d f(\ox)(w)- \la\bar v,w\ra\leq 0\big\}$ of an lsc function $w\mapsto \d f(\ox)(w) - \la \ov, w\ra$ (cf. \cite[Proposition~7.4]{rw}). Thus, we conclude that $K_f(\ox, \ov)$ is a closed set and complete the proof of (a).


To justify (b), we infer from \cite[Theorem~8.30]{rw} and subdifferential regularity of $f$ at $\ox$ that
\begin{equation*}
{K_f}(\ox,\bar v)=\big\{w\in \X\,\big|\,\la\bar v,w\ra=\sup_{v\in \sub f(\ox)}\la v, w\ra\big\} = N_{\sub f(\ox)}(\ov),
\end{equation*}
where the second equality is due to convexity of $\sub f(\ox) = \Hat\sub f (\ox)$.
This proves (b) and completes the proof. 
\end{proof}

We proceed with a consequence of subdifferential regularity and strict proto-differentiability, which is important for our characterization of the latter concept for ${\cal C}^2$-decomposable functions in Section~\ref{chain}.

\begin{Proposition} \label{pedcon} Assume that $f:\X\to \oR$ is prox-regular and subdifferentially continuous at $\ox$ for  $\ov\in   \sub f(\ox)$ and that $\sub f$ is strictly proto-differentiable at $\ox$ for $\ov$. 
Assume further that $f$ is subdifferentially regular at $\ox$ and   that   $\dom \d^2 f(\ox,\ov)= K_f(\ox,\ov)$. Then, we have $\ov \in \ri   \sub f(\ox)$.
\end{Proposition}
\begin{proof} It follows from Proposition~\ref{proxcri} that  $K_f(\ox,\ov)= N_{ \sub f(\ox)}(\ov)$.  By Proposition~\ref{glp}(b), $\dom \d^2 f(\ox,\ov)=K_f(\ox,\ov)$ must be a linear subspace of $\X$. 
It also results from subdifferential regularity of $f$ at $\ox$ that  $\sub f(\ox)=\Hat\sub f(\ox)$, which tells us that $\sub f(\ox)$ is convex. Thus, it follows from a well-known fact from convex analysis (cf. \cite[Proposition~2.51]{mn}) that $K_f(\ox,\ov)$ being a linear subspace is equivalent to  $\ov\in  \ri   \sub f(\ox) $. This completes the proof.
\end{proof}

Note that by Proposition~\ref{proxcri}(b), one may expect  to assume the condition $\cl\big(\dom \d^2 f(\ox,\ov)\big)= K_f(\ox,\ov)$ instead of  $\dom \d^2 f(\ox,\ov)= K_f(\ox,\ov)$ in the result above.
They are, however, equivalent in the setting of Proposition~\ref{pedcon}, since $\dom \d^2 f(\ox,\ov)$ must be a linear subspace.

Proposition~\ref{pedcon} tells us that strict proto-differentiability of subgradient mappings necessitates that in the presence of subdifferential regularity, the subgradient under consideration must satisfy a certain relative interior condition.
We recently demonstrated in \cite{HJS22, HaS22} that the latter condition is indeed equivalent to strict proto-differentiability of subgradient mappings for polyhedral functions and certain  composite functions. Using Proposition~\ref{pedcon}, we 
will justify a similar characterization for ${\cal C}^2$-decomposable functions in Section~\ref{chain}. 

While the assumptions in Proposition~\ref{pedcon} hold for many   classes of functions, important for applications to constrained and composite optimization,
the assumption about the domain of the second subderivative requires more elaboration. In the framework of Proposition~\ref{pedcon}, it follows from Proposition~\ref{proxcri} that 
$\dom \d^2 f(\ox,\ov)\subset K_f(\ox,\ov)$. 
Equality in the latter inclusion was studied in 
 \cite[Proposition~3.4]{ms20}. To present it, take $w\in \X$ with $\d f(\ox)(w)$ finite and define 
the parabolic subderivative of $f$ at $\ox$ for $w$ with respect to $z$  by 
\begin{equation*}\label{lk02}
\d^2 f(\bar x)(w\verl z)= \liminf_{\substack{
   t\searrow 0 \\
  z'\to z
  }} \dfrac{f(\ox+tw+\frac{1}{2}t^2 z')-f(\ox)-t\d f(\ox)(w)}{\frac{1}{2}t^2}.
\end{equation*}
It was shown  in \cite[Proposition~3.4]{ms20} that the condition $\dom \d^2 f(\bar x)(w\verl \cdot)\neq \emptyset$ for all $w\in K_f(\ox,\ov)$ yields
$\dom \d^2 f(\ox,\ov)= K_f(\ox,\ov)$. 
It is worth mentioning that the latter condition imposed on $\dom \d^2 f(\bar x)(w\verl \cdot)$ is satisfied whenever $f$ is parabolically epi-differentiable at 
$\ox$ for any $w\in K_f(\ox,\ov)$ in the sense of \cite[Definition~13.59]{rw}. 
Note that parabolic epi-differentiability has been studied extensively in \cite{mms,ms20,ms23} and holds for many important classes of functions
including convex piecewise linear-quadratic functions (cf. \cite[Theorem~13.67]{rw}),  ${\cal C}^2$-decomposable functions (cf. \cite[Theorem~6.2]{HaS23}), and spectral functions (cf. \cite[Theorem~4.7]{ms23}).  

\begin{Corollary}Assume that $f:\X\to \oR$ is prox-regular and subdifferentially continuous at $\ox$ for  $\ov\in   \sub f(\ox)$ and that $\sub f$ is strictly proto-differentiable at $\ox$ for $\ov$. 
Assume further that $f$ is subdifferentially regular at $\ox$ and   that $\dom \d^2 f(\bar x)(w\verl \cdot)\neq \emptyset$ for any $w\in K_f(\ox,\ov)$, then 
one has  $\ov \in \ri   \sub f(\ox)$.
\end{Corollary}
\begin{proof} This results from the discussion above, Proposition~\ref{pedcon}, and  \cite[Proposition~3.4]{ms20}.
\end{proof}

\begin{Proposition} \label{crit} Assume that $f:\X\to \oR$ is prox-regular and subdifferentially continuous at $\ox$ for  $\ov\in   \sub f(\ox)$ and that $\sub f$ is  proto-differentiable at $\ox$ for $\ov$. 
Assume further that $f$ is subdifferentially regular at $\ox$ and   that   $\cl \big(\dom \d^2 f(\ox,\ov)\big)= K_f(\ox,\ov)$. Then, we have 
$$
D(\sub f)(\ox,\ov)(0)=N_{K_f(\ox,\ov)}(0)=K_f(\ox,\ov)^*.
$$
\end{Proposition}
\begin{proof} Since $\sub f$ is  proto-differentiable at $\ox$ for $\ov$, it follows from \cite[Theorem~13.40]{rw} that $D( \sub f)(\ox,\ov)(0)= \sub (\sm \d^2 f(\ox,\ov))(0)$.
Using a similar argument as the one in the proof of Proposition~\ref{glp}(b), we find $\rho\ge 0$ such that 
the function $\ph$, defined by $\ph(w)=\d^2 f(\ox,\ov)(w)+\rho \|w\|^2$ for $w\in \X$, is a proper convex function with $\ph(0)=0$.
Moreover, by \cite[Proposition~2.1(iii)]{ms20}, we know that $\d^2 f(\ox,\ov)(w)\ge -r \|w\|^2$ for any $w\in \X$, where $r$ is taken from \eqref{prox}. 
Choosing $\rho$ sufficiently large, we can then assume without loss of generality that $\ph(w)\ge 0$ for any $w\in \X$. 
It follows from the definition of $\ph$ that $\dom \ph= \dom \d^2 f(\ox,\ov)$. 
 Take now $\eta\in D( \sub f)(\ox,\ov)(0)$ and conclude from  $ \sub (\sm\ph)(0)= \sub (\sm \d^2 f(\ox,\ov))(0)$ 
that $\eta\in \sub (\sm\ph)(0)$. Since $\ph$ is convex, we obtain for any $w\in \dom \ph$ and any $t>0$ that 
$$
\la \eta, tw-0\ra \le \sm\ph(tw)-\sm\ph(0)= \sm \d^2 f(\ox,\ov)(tw)+\sm{\rho} \|tw\|^2= t^2( \sm \d^2 f(\ox,\ov)(w)+\sm{\rho} \|w\|^2),
$$
 where the last equality results from the fact the second subderivative is positively homogeneous of degree $2$. Dividing both sides by $t$ and letting then $t\to 0$
 bring us to $\la \eta,w\ra\le 0$ for any $w\in \dom \ph$. Since $\cl (\dom \ph)=K_f(\ox,\ov)$, we arrive at  $\eta\in K_f(\ox,\ov)^*$. 
 
 To justify the opposite inclusion, take $\eta\in K_f(\ox,\ov)^*$ and observe that 
 $$
 \la \eta, w-0\ra \le 0\le  \sm\ph(w)= \sm\ph(w)- \sm\ph(0), \quad \mbox{for all}\;\; w\in K_f(\ox,\ov).
 $$
 This confirms that $\eta\in  \sub (\sm\ph)(0)= \sub (\sm \d^2 f(\ox,\ov))(0)$ and hence $\eta\in D( \sub f)(\ox,\ov)(0)$, which completes the proof. 
\end{proof}

Recall that a set-valued mapping $F:\X\tto \Y$ is said to be metrically subregular at $x$ for  $y \in  F(x)$ if there are a constant $\ell \ge 0$
and a neighborhood $U$ of $x$ such that the estimate $\dist(x',F^{-1}(y))\le \ell\, \dist (y,F(x'))$ holds for any $x'\in U$. 

\begin{Proposition}\label{dom_d2} Assume that $f: \X \to \oR$ is a proper lsc convex function and $(\ox,\ov)\in \gph \sub f$. Then we always have 
$$
T_{\sub f^*(\ov)}(\ox)\subset \ker \d^2 f(\bar x , \ov):= \big\{w\in \X\,|\, \d^2 f(\bar x , \ov)(w)=0\big\}.
$$
Equality holds provided that $\sub f$ is metrically subregular at $\ox$ for $\ov$. In this case, if $\sub f$ is proto-differentiable at $\ox$ for $\ov$, 
the latter equality amounts to $\cl\big(\dom \d^2 f^*(\ov,\ox)\big)=K_{f^*}(\ov,\ox)$.
\end{Proposition}
\begin{proof} Take $w\in T_{\sub f^*(\ov)}(\ox)$ and conclude by definition that there are sequences $t_k\searrow 0$ and $w^k\to w$ such that $\ox+t_kw^k\in \sub f^*(\ov)$ for all $k$.
This yields $\ov\in \sub f(\ox+t_kw^k)$, which brings us to 
$$
\frac{f(\ox+t_kw^k) -f(\ox)-t_k \la \ov, w^k\ra}{\tfrac{1}{2}t_k^2}\le 0 \quad \mbox{for all}\;\;k.
$$
Thus, we arrive at $ \d^2 f(\bar x , \ov)(w)\le 0$. Since $f$ is convex, we always have $ \d^2 f(\bar x , \ov)(w)\ge 0$, which proves that $w\in  \ker \d^2 f(\bar x , \ov)$. 

Assume now that $w\in \ker \d^2 f(\bar x , \ov)$ and that $\sub f$ is metrically subregular at $\ox$ for $\ov$. By  \cite[Theorem~3.3]{ag}, the latter   is equivalent to saying that $f$ enjoys  the quadratic growth condition at $\ox$ for $\ov$, meaning that there are 
positive constants $\ve$ and $\ell$ such that  
\begin{equation}\label{qgc}
f(x)\ge f(\ox)+\la \ov, x-\ox\ra + \frac{\ell}{2}\dist^2(x,(\sub f)^{-1}(\ov)) \quad \textrm{ for all }\; x\in \B_\ve(\ox).
\end{equation}
Take sequences $t_k\searrow 0$ and $w^k\to w$ for which we have $\Delta_{t_k}^2 f(\bar x , \ov)(w^k)\to  \d^2 f(\bar x , \ov)(w)$.
Employing \eqref{qgc} tells us that 
$$
\Delta_{t_k}^2 f(\bar x , \ov)(w^k) \ge \frac{\ell \dist^2(\ox+t_kw^k,(\sub f)^{-1}(\ov))}{ t_k^2} \ge  0
$$
for all sufficiently large $k$. Passing to the limit, coupled with $  \d^2 f(\bar x , \ov)(w)=0$,  implies that   $\dist(\ox+t_kw^k,(\sub f)^{-1}(\ov)) =o(t_k)$,
which results in $\ox+t_kw^k +o(t_k)\in (\sub f)^{-1}(\ov) = \sub f^*(\ov)$ for all sufficiently large $k$. This clearly tells us that 
$w\in T_{\sub f^*(\ov)}(\ox)$ and hence proves the claimed equality.
The final claim results from \cite[equation~(4.28)]{r23}.
\end{proof}

Below, we show that strict proto-differentiability 
of subgradient mappings can ensure the opposite implication in Proposition~\ref{dom_d2} holds as well and hence gives a characterization of metric subregularity of subgradient mappings of convex functions. 

\begin{Theorem}\label{msr5} Assume that $f:\X\to \oR$ is a proper lsc convex and $(\ox,\ov)\in   \gph \sub f$ and that $\sub f$ is  strictly proto-differentiable at $\ox$ for $\ov$. 
Then the subgradient mapping $\sub f$ is metrically subregular at $\ox$ for $\ov$ if and only if one of the equivalent conditions $ \ker \d^2 f(\bar x , \ov)=T_{\sub f^*(\ov)}(\ox)$
or  $\dom \d^2 f^*(\ov,\ox)= K_{f^*}(\ov,\ox)$ is satisfied.
\end{Theorem}

\begin{proof} Since  $\sub f$ is strictly proto-differentiable at $\ox$ for $\ov$, it follows from $\sub f^*=(\sub f)^{-1}$
that $\sub f^*$ is strictly proto-differentiable at $\ov$ for $\ox$. By Proposition~\ref{glp}(b), $\dom \d^2 f^*(\ov,\ox)$ is a linear subspace of $\X$ and thus is closed. The equivalence between $\ker \d^2 f(\bar x , \ov)=T_{\sub f^*(\ov)}(\ox)$
and  $\dom \d^2 f^*(\ov,\ox)= K_{f^*}(\ov,\ox)$ then results from Proposition~\ref{dom_d2}.

Suppose that  $\dom \d^2 f^*(\ov,\ox)= K_{f^*}(\ov,\ox)$ but  $\sub f$ is not metrically subregular at $\ox$ for $\ov$. 
Thus, we  find a sequence  $x^k\to \ox$ as $k\to \infty$ such that 
$$
k\dist (\ov, \sub f(x^k))< \dist (x^k, (\sub f)^{-1}(\ov))=\dist(x^k, \sub f^*(\ov)),
$$
where the last equality results from the fact that $(\sub f)^{-1}=\sub f^*$ (cf. \cite[Proposition~11.3]{rw}). Since both sets $\sub f(x^k)$ and $\sub f^*(\ov)$ are closed and convex, we find 
$v^k=P_{ \sub f(x^k)}(\ov)$ and $\tilde x^k=P_{ \sub f^*(\ov)}(x^k)$ such that $\dist (\ov, \sub f(x^k))=\|v^k-\ov\|$ and $\dist(x^k, \sub f^*(\ov))=\|x^k-\tilde x^k\|$, respectively. Set $t_k:= \|x^k-\tilde x^k\|$
and observe from the inequality above that $v^k-\ov=o(t_k)$. Passing to a subsequence, if necessary, we can assume that $\{(x^k-\tilde x^k)/t_k\b$ converges to some $u\in \X$ with $\|u\|=1$.
Since 
$$
(\tilde x^k+t_k\big((x^k-\tilde x^k)/t_k\big), \ov+t_k\big( (v^k-\ov)/t_k\big)= (x^k, v^k)\in \gph \sub f,
$$
we obtain $0\in \widetilde D(\sub f)(\ox,\ov)(u)$ via the definition of the strict graphical derivative.  
We have $\widetilde D(\sub f)(\ox,\ov)(u)=  D(\sub f)(\ox,\ov)(u)$ by strict proto-differentiability of $\sub f$  at $\ox$ for $\ov$, and therefore $0\in   D(\sub f)(\ox,\ov)(u)$. The latter amounts to 
$$
u\in D(\sub f)^{-1}(\ov,\ox)(0)=D(\sub f^*)(\ov,\ox)(0).
$$
On the other hand, strict proto-differentiability of $\sub f$  at $\ox$ for $\ov$ is equivalent to  that of $\sub f^*$ at $\ov$ for $\ox$. 
Since $f^*$ is convex, it is subdifferentially regular at $\ov$. Combining these and  Proposition~\ref{crit}
leads us to 
\begin{equation}\label{ueq1}
u\in D(\sub f^*)(\ov,\ox)(0)={K_{f^*}(\ov,\ox)}^*. 
\end{equation}
Moreover, it follows from $\tilde x^k=P_{ \sub f^*(\ov)}(x^k)$ that $(x^k-\tilde x^k)/t_k\in N_{\sub f^*(\ov)}(\tilde x^k)$, which implies via Proposition~\ref{proxcri}(b) that 
$$
u\in N_{\sub f^*(\ov)}(\ox)=K_{f^*}(\ov,\ox).
$$
This, coupled with \eqref{ueq1}, results in $u=0$, a contradiction,  which completes the proof. 
\end{proof}

Strict proto-differentiability, assumed in Theorem~\ref{msr5}, will be characterized  for ${\cal C}^2$-decomposable functions  in Section~\ref{chain}.
Below, we provide an example of a class of functions for which the assumption on the domain of the second subderivative in Theorem~\ref{msr5}  is always satisfied. 

\begin{Example} \label{spectral} 
Let $\S^n$  stand for the space of all real $n\times n$ symmetric  matrices
equipped with the inner product 
$$
\la X, Y\ra=\tr(XY),\quad X,Y\in \S^n.
$$
The induced Frobenius norm of $X\in \S^n$ is defined by $\|X\|=\sqrt{\tr(X^2)}$.
Recall that  $f:\S^n\to \oR$  is a spectral function if it is  orthogonally invariant, meaning that  for any $X \in \S^n$ and any $n\times n$ orthogonal matrix $U$, 
we have $ f (X) = f (U^\top X U )$.
 It follows from \cite[Proposition~4]{l99} that any spectral function $f$ can be expressed in the composite form 
 \begin{equation*}\label{spec}
f(X):=(\th\circ\lm)(X), \;\; X\in \S^n,
\end{equation*}
where $\th:\R^n\to  \oR$ is a permutation-invariant function on $\R^n$, called symmetric, and $\lm$ is the mapping assigning to each matrix $X\in \S^n$ the vector $\big(\lm_1(X),\ldots,\lm_n(X)\big)$
 of its eigenvalues arranged in nonincreasing order. Suppose that  the symmetric function $\th:\R^n\to \oR$  is polyhedral, meaning $\epi \th$ is a polyhedral convex set. Consider the spectral function 
 $f=\th\circ\lm$. This selection of $\th$ allows to covers important examples of eigenvalue functions such as the maximum eigenvalue function and  the sum of the first $k$ ($1\le k\le n$) largest eigenvalues of 
 a matrix. 
 It follows from \cite[Theorem~5.2.2]{bl} that $f^*=\th^*\circ\lm$. According to \cite[Theorem~11.14]{rw}, $\th^*$ is polyhedral. Appealing now to \cite[Corollary~5.8]{ms23} tells us that 
 if $(X,Y)\in \gph \sub f$, then we always have $\dom \d^2 f^*(Y,X)= K_{f^*}(Y,X)$.  
 \end{Example}

We continue  with a relationship between graphical derivative and coderivative of subgradient mappings of strictly proto-differentiable functions.

\begin{Theorem}\label{thm:gdcod}
Assume that $f:\X\to \oR$ is prox-regular and subdifferentially continuous at $\ox$ for  $\ov\in   \sub f(\ox)$ and that $\sub f$ is strictly proto-differentiable at $\ox$ for $\ov$. 
Then, one has 
\begin{equation}\label{gdcod}
D( \sub f)(\ox,\ov)(w)=D^*( \sub f)(\ox,\ov)(w)=A(w)+N_{\K}(w)\quad \mbox{for all}\;\; w\in \X,
\end{equation}
where $\K$ and $A$ are taken from Proposition~{\rm\ref{glp}(b)}. 
\end{Theorem}

\begin{proof} Strict proto-differentiability of $\sub f$   at $\ox$ for $\ov$ implies  
proto-differentiability of $\sub f$   at $\ox$ for $\ov$, which together with \cite[Theorem~13.40]{rw} illustrates that $f$ is twice epi-differentiable at $\ox$ for $\ov$. 
 By \cite[Theorem~13.57]{rw}, we get the inclusion $D( \sub f)(\ox,\ov)(w)\subset D^*( \sub f)(\ox,\ov)(w)$. To obtain the opposite inclusion, take $\eta\in D^*( \sub f)(\ox,\ov)(w)$, which means that $(\eta,-w)\in N_{\gph \sub f}(\ox,\ov)$. By Proposition~\ref{glp}(c), we have $(\eta,-w)\in \rN_{\gph \sub f}(\ox,\ov)$.
It also follows from Proposition~\ref{glp}(b) that $ \d^2 f(\ox,\ov)$ is a generalized quadratic form on $\X$ given by \eqref{gqf1}. This representation  brings us to 
\begin{equation}\label{dgr}
D( \sub f)(\ox,\ov)(w)= \sub (\sm \d^2 f(\ox,\ov))(w)= A(w)+N_{\K}(w)\quad \mbox{for all}\;\; w\in \X,
\end{equation}
where the first equality results from \cite[Theorem~13.40]{rw}. 
This clearly  confirms that 
$$
T_{\gph \sub f}(\ox,\ov)=\gph D( \sub f)(\ox,\ov)=\big\{(w,\eta)\,\big|\; 
w\in \K, \, \eta - A(w) \in \K^\perp\big\},
$$
which leads us via \cite[Corollary~11.25(d)]{rw} to 
\begin{equation*}
\rN_{\gph \sub f}(\ox,\ov)=\big(T_{\gph \sub f}(\ox,\ov)\big)^*=\big\{ 
(\xi-A(\nu), \nu)\,
\big|\;(\xi,\nu)\in 
   \K^\perp\times \K\big\}.
\end{equation*}
Since $(\eta,-w)\in \rN_{\gph \sub f}(\ox,\ov)$,  there exists $(\xi,\nu)\in \K^\perp\times \K$ such that 
$$
\xi-A(\nu)=\eta\quad \mbox{and}\quad \nu=-w.
$$
These equations imply  that $\eta-A(w)=\xi\in \K^\perp$ and $w\in \K$, which, together with the fact that $\K$ is a linear subspace of $\X$ and \eqref{dgr}, shows that $\eta\in D( \sub f)(\ox,\ov)(w)$, as desired, and completes the proof.
\end{proof}

As mentioned in the proof of Theorem~\ref{thm:gdcod}, the inclusion $D( \sub f)(\ox,\ov)(w)\subset D^*( \sub f)(\ox,\ov)(w)$ always holds under proto-differentiability of $\partial f$; see  \cite[Theorem~13.57]{rw}.
It is also not hard to see that the latter inclusion can be strict;  consider the function $f(x)=x^2/2$ for $x\ge 0$ and $f(x)=-x^2/2$ for $x\le 0$.  Thus,  equality requires to impose more assumptions. 

We recently characterized the relationship in \eqref{gdcod} via  a relative interior condition  for polyhedral 
functions in \cite[Corollary~3.7]{HJS22} and for certain  composite functions in \cite[Theorem~3.12]{HaS22}. Moreover, it was shown therein that such a 
condition is equivalent to strict   proto-differentiability of subgradient mappings under consideration. Whether or not a similar result can be obtained in the general setting of Theorem~\ref{thm:gdcod} remains as an open question.
We should also mention that it was shown in \cite[Theorem~4]{lw} that the same equality in \eqref{dgr} holds for ${\cal C}^2$-partly smooth functions in the sense of \cite[Definition~2.7]{Lew02}. 
We should add that ${\cal C}^2$-partly smooth functions are always strictly proto-differentiable and so \cite[Theorem~4]{lw} can be covered by Theorem~\ref{thm:gdcod}. The proof of 
the latter result is beyond the scope of this paper and will be appeared in our forthcoming paper \cite{HaS24}. 

We close this section by showing that our observation in \eqref{gdcod} has an interesting consequence about local minimizers of a function. To present it, we have to recall the concept of tilt-stable local minimizers of a function. 
Given a function $f:\X\to \oR$ and $\ox\in \X$ with $f(\ox)$ finite, we say that 
 $\ox$ is  a tilt-stable local minimizer of $f$
if there exist neighborhoods $U$ of $\ox$ and $V$ of $\ov=0$ such that  the optimal solution mapping 
$$
v\mapsto \argmin_{x\in U}\{ f(x)-\la v,x-\ox\ra\}
$$
is single-valued and Lipschitz continuous on $V$ and its value at $\ov$ is $\{\ox\}$.  
Tilt-stability was introduced in \cite{pr1} and has been studied extensively for various classes of constrained and composite optimization problems; see \cite{mn14,pr1}.
According to \cite[Theorem~3.2]{mn14}, tilt-stability of a local minimizer $\ox$ of $f$ yields the quadratic growth condition of $f$ at $\ox$: there exist a constant  $\ell\ge 0$ and a neighborhood $U$ of $\ox$ such that
$$
f(x)\ge f(\ox)+\frac{\ell}{2}\|x-\ox\|^2\quad \mbox{for all}\;\; x\in U.
$$
Below, we are going to show that 
under strict proto-differentiability of subgradient mappings the opposite implication is also valid. 

\begin{Corollary}\label{tilt}
Assume that $f:\X\to \oR$ is prox-regular and subdifferentially continuous at $\ox$ for  $\ov:=0\in   \sub f(\ox)$ and that $\sub f$ is strictly proto-differentiable at $\ox$ for $\ov$. 
Then, the following properties are equivalent: 
\begin{itemize}[noitemsep,topsep=2pt]
\item [ \rm {(a)}] $\ox$ is a tilt-stable local minimizer of $f$; 
\item [ \rm {(b)}]  $f$ enjoys the quadratic growth condition at $\ox$;
\item [ \rm {(c)}]   for any nonzero $w\in \dom D( \sub f)(\ox,\ov)$ and any $\eta\in D( \sub f)(\ox,\ov)(w)$,  one has $\la w,\eta\ra >0$.
\end{itemize}
\end{Corollary}
\begin{proof} 
Strict proto-differentiability of $\sub f$   at $\ox$ for $\ov$ clearly implies 
proto-differentiability of $\sub f$   at $\ox$ for $\ov$, which together with \cite[Theorem~13.40]{rw} illustrates that $f$ is twice epi-differentiable at $\ox$ for $\ov$. 
The equivalence of (b) and (c) was established in \cite[Theorem~3.7]{chnt}. It follows from \cite[Theorem~1.3]{pr1} that (a) is equivalent  to the condition 
$\la w,\eta\ra >0$ for any $w\in \big(\dom D^*( \sub f)(\ox,\ov)\big)\setminus \{0\}$ and any $\eta\in D^*( \sub f)(\ox,\ov)(w)$. Combining this with Theorem~\ref{thm:gdcod}
tells us that (a) and (c) are equivalent, which completes the proof. 
\end{proof}

Using a different approach, Lewis and Zhang demonstrated in \cite[Theorem~6.3]{LeZ13} the equivalence of the properties in Corollary~\ref{tilt}(a) and (b) for ${\cal C}^2$-partly smooth functions. As pointed out earlier, 
the latter functions  are  strictly proto-differentiable, which allows us to obtain their result using our established theory in this section.


\section{Regularity Properties of Generalized Equations}\label{sec04}

In this section, we aim to study important stability properties of solution mappings to a class of generalized equations under the strict proto-differentiability assumption. As Section~\ref{sec3} may suggest, we should expect 
stronger properties to be satisfied under the latter assumption. We begin with an important question about the relationship between two important stability properties of solution mappings of generalized equations, namely metric regularity and strong metric regularity. 
Given a differentiable mapping $\psi:\X \to \X$ and a proper function $f:\X\to\oR$, we mainly focus on the generalized equation 
\begin{equation}\label{ge}
0\in \psi(x)+\partial f(x),
\end{equation}
to which we associate  the set-valued mapping $G:\X \tto \X$, defined  by
\begin{equation}\label{G}
G(x) = \psi(x) +\partial f(x), \quad x\in \X,
\end{equation}
and the  solution mapping $S: \X \tto \X$, defined  by 
\begin{equation}\label{sol}
S(y) := G^{-1}(y) = \big\{ x\in \X\, \big|\, y\in \psi(x) +\partial f(x)\big\}, \quad y\in \X.
\end{equation}
We will assume further that the function $f$ in \eqref{ge} is prox-regular at the point under consideration. 
Recall that a set-valued mapping $F:\X \tto \Y$ is said to be {\em metrically regular} at $\ox$ for $\oy\in F(\ox)$ if there exist a positive constant $\kappa$  and  neighborhoods $U$ of $\ox$ and $V$ of $\oy$ such that the estimate
\begin{equation}\label{mr8}
\dist \big(x, F^{-1}(y)\big)\leq \kappa\, \dist\big(y, F(x)\big)
\end{equation}
holds for all $(x, y)\in U\times V$. When the estimate \eqref{mr8} holds for any $(x,y)\in \X \times \Y$, we say that $F$ is {\em globally}  metrically regular. 
The mapping is said to be {\em strongly metrically regular} at $\ox$ for $\oy$ if $F^{-1}$ admits a Lipschitz continuous single-valued localization around $\oy$ for $\ox$, which means that there exist neighborhoods $U$ of $\ox$ and $V$ of $\oy$ such that the mapping $y\mapsto F^{-1}(y)\cap U$ is single-valued and Lipschitz continuous on $V$. According to \cite[Proposition~3G.1]{DoR14}, strong metric regularity of $F$ at $\ox$ for $\oy$ amounts to $F$ being metrically regular at $\ox$ for $\oy$ and the inverse mapping $F^{-1}$ admitting a single-valued localization around $\oy$ for $\ox$.

When $f=\dd_C$ in \eqref{ge} with $C$ being a polyhedral convex set, the generalized 
equation \eqref{ge} presents an example of  variational inequalities. In this case, the seminal paper \cite{dr96} revealed for the first time that strong metric regularity 
and metric regularity of the solution mapping $S$ in \eqref{sol}  are equivalent. In this section, we are going to justify the same equivalence 
for the generalized equation in \eqref{ge} when $\sub f$ is   strictly proto-differentiable.  

While metric regularity of the mapping $G$ from \eqref{G}  can be studied using the coderivative criterion (cf.  \cite[Theorem~3.3(ii)]{Mor18}) in general, we show below that 
in the presence of   strict proto-differetiability, such a characterization has   simpler forms. 

\begin{Proposition}[point-based criteria for metric regularity]\label{thm:MR}
Assume that $\ox$ is a solution to the generalized equation in \eqref{ge} in which $\psi$ is strictly differentiable at $\ox$ and  $f$ is both prox-regular and subdifferentially continuous at $\ox$ for $\ov:=-\psi(\ox)$ and 
that $\partial f$ is strictly proto-differentiable at the point $\ox$ for  $\ov$. Take the linear subspace $\K$ and the linear operator $A$
from \eqref{gdcod}. Then the following properties are equivalent:
\begin{enumerate}[noitemsep,topsep=2pt]
\item  the mapping $G$ from \eqref{G} is metrically regular at $\ox$ for $0$;
\item $\big\{w\in \X\, \big|\, (\nabla \psi(\ox)+A)^*(w)\in \K^{\perp}\big\}\cap \K=\{0\}$;
\item $(\nabla\psi(\ox)+A)(\K) +\K^\perp = \X$;
\item $|DG(\ox, 0)^{-1}|^-$ is finite, where the inner norm $|DG(\ox, 0)^{-1}|^-$  is defined by
\begin{equation*}\label{mr5}
|DG(\ox, 0)^{-1}|^- := \sup_{\|u\|\leq 1}\inf_{w\in DG(\ox, 0)^{-1}(u)}\|w\|
\end{equation*}
with   convention $\inf_{w\in \emp}\|w\|=\infty$;
\item $DG(\ox, 0) $ is surjective, meaning that 
$$
\rge DG(\ox, 0):=\big\{u\in \X|\; \exists \,  w\in \X\;\,\mbox{with}\;\; u\in DG(\ox, 0)(w)\big\}=\X;
$$
\item $DG(\ox, 0)$ is globally metrically regular.
\end{enumerate}
\end{Proposition}

\begin{proof} To prove the equivalence of (a) and (b), we begin with calculating $D^*G(\ox, 0)$. 
Applying the  sum rule for the coderivative from \cite[Exercise~10.43(b)]{rw} to  $G = \psi +\partial f$ and using  the formula for $D^*(\partial f)(\ox, \ov)$ in \eqref{gdcod}, we obtain for all $w\in \X$ that
\begin{equation}\label{cosum}
D^*G(\ox, 0)(w) = \nabla \psi(\ox)^*(w) + D^*(\partial f)(\ox, \ov)(w)
= (\nabla \psi(\ox) + A)^*(w) +N_\K(w).
\end{equation}
Since $\K$   is a linear subspace of $\X$, we then deduce from the above calculation that
\begin{equation*}
D^*G(\ox, 0)^{-1}(0) :=  \big\{w\in \X\, \big|\, 0\in D^*G(\ox, 0)(w)\big\}= \big\{w\in \X\, \big|\, (\nabla \psi(\ox) + A)^*(w) \in \K^\bot\big\}\cap \K.
\end{equation*}
By  \cite[Theorem~3.3(ii)]{Mor18}, metric regularity of $G$ at $\ox$ for $0$ amounts to  $D^*G(\ox, 0)^{-1}(0)=\{0\}$. 
Combining it with the discussion above yields the equivalence of (a) and (b). 

Turing to the equivalence of (b) and (c),  observe first  via \cite[Corollary~11.25(c)]{rw} that
\begin{equation*}
\big\{w\in \X\, \big|\, (\nabla \psi(\ox)+A)^*(w)\in \K^{\perp}\big\}^\bot =(\nabla\psi(\ox)+A)(\K).
\end{equation*}
Taking the orthogonal complements of both sides of either (b) or (c) and using the above equality tell us  that (b) and (c) are equivalent. 

To prove the equivalence of (c) and (e),  we proceed with calculating $DG(\ox, 0)$. Since $\psi$ is differentiable at $\ox$,  the  sum rule for the graphical derivative from \cite[Exercise~10.43(a)]{rw},  together with \eqref{gdcod}, gives us
\begin{equation}\label{mr2.3}
DG(\ox, 0)(w) = \nabla \psi(\ox)(w) + D(\partial f)(\ox, \ov)(w) = (\nabla \psi(\ox) + A)(w) + N_\K(w), \quad w\in \X.
\end{equation}
Thus, we get    
\begin{equation}\label{mr2.1}
\rge DG(\ox, 0) = (\nabla\psi(\ox)+A)(\K ) +\K^\perp,
\end{equation}
which clearly implies that  $\rge DG(\ox, 0)=\X$  if and only if (c) holds. This verifies the equivalence of (c) and (e). 

According to \eqref{mr2.3}, $\gph DG(\ox, 0)$ is closed and convex due to  $\K$ being a linear subspace of $\X$. It follows then from  
 \cite[Proposition~4A.6]{DoR14}  that $|DG(\ox, 0)^{-1}|^-<\infty$  amounts to $DG(\ox, 0)$ being surjective, which proves the equivalence of (d) and (e). 
Finally, observe that  (e)  amounts to the condition $0\in \inte\big( \rge DG(\ox,0)\big)$ due to the fact that 
$\rge DG(\ox,0)$ is indeed a linear subspace of $\X$; see \eqref{mr2.1}. This, combined with $0\in DG(\ox,0)(0)$ and \cite[Theorem~5B.4]{DoR14},
implies that (e) is equivalent to  the mapping  $DG(\ox,0)$ being metrically regular at $0\in \X$ for $0\in \X$. According to \cite[Theorem~5.9(a)]{io}, the latter amounts to (f),
confirming that (e) and (f) are equivalent. This completes the proof.
\end{proof}

To proceed, we recall the following   sufficient conditions for strong metric regularity of a set-valued mapping, taken from   \cite[Theorem~4D.1]{DoR14}:  a set-valued mapping   $F:\X \tto \Y$  is strongly metrically regular at $\ox$ for $\oy\in F(\ox)$ provided that its graph is locally closed at $(\ox, \oy)$  and 
the conditions 
 \begin{equation}\label{mr9}
0\in \widetilde D F(\ox, \oy)(w)\; \Longrightarrow\; w=0
\end{equation}
 and 
\begin{equation}\label{mr7}
\ox \in \liminf_{y\to \oy} F^{-1}(y)
\end{equation}
hold.

 \begin{Theorem}\label{smrch}  
Assume that $\ox$ is a solution to the generalized equation in \eqref{ge} in which $\psi$ is strictly differentiable at $\ox$ and  $f$ is both prox-regular and subdifferentially continuous at $\ox$ for $\ov:=-\psi(\ox)$.  
 Then  the following properties are equivalent:
 \begin{enumerate}[noitemsep,topsep=2pt]
\item the mapping $G$, taken from \eqref{G}, is metrically regular at $\ox$ for $0$ and  $\partial f$ is strictly proto-differentiable at  $\ox$ for  $\ov$;
\item the solution mapping $S$ from \eqref{sol} has a Lipschitz continuous single-valued localization $s$ around $0\in \X$ for $\ox$, which  is strictly differentiable at  $0$. 
\end{enumerate} 
Moreover, if {\rm(}a{\rm)} holds, then the derivative of $s$ at $0$ can be calculated by   
\begin{equation}\label{derS}
\nabla s(0) = \big(P_\K\circ(\nabla\psi(\ox) + A)\vert_\K\big)^{-1}\circ P_\K,
\end{equation}
where the linear mapping $A$  and the linear subspace $\K$ are taken from  \eqref{gdcod},  where $P_\K:\X \to\X$ is the   projection mapping onto  $\K$, and 
where $(\nabla\psi(\ox) + A)\vert_\K$ stands for the restriction of the linear mapping $\nabla\psi(\ox)+A$ to $\K$.
\end{Theorem}

\begin{proof} 
Suppose that (a) holds. We begin by proving that $G$ is strongly metrically regular at $\ox$ for $0$
by validating \eqref{mr9} and \eqref{mr7}. To this end, it can be seen via the definition of $G$ in \eqref{G} that the graph of $G$ is closed. 
Moreover, the distance estimate in \eqref{mr8}, adopted for metric regularity of $G$ at $\ox$ for $0$, clearly yields \eqref{mr7}. It remains to justify \eqref{mr9}.  Making use of strict proto-differentiability of $\partial f$ at $\ox$ for $\ov$, we conclude via \cite[Proposition~5.1]{HJS22} that $G$ is strictly proto-differetiable at $\ox$ for $0$, which results in  $\widetilde D G(\ox, 0)(w)=DG(\ox, 0)(w)$
for any $w\in \X$. 
Also, we know that $D G(\ox, 0)$ enjoys the representation in \eqref{mr2.3}. Moreover, we deduce from metric regularity of $G$ at $\ox$ for $0$ and Proposition~\ref{thm:MR}(c) that
the condition 
\begin{equation}\label{mr2.5}
(\nabla\psi(\ox)+A)(\K) +\K^\perp = \X
\end{equation}
is satisfied.
Set $\widehat K := (\nabla\psi(\ox)+A)(\K)$. Obviously, we have $\dim \widehat K\leq \dim \K$.
We now show that $\widehat K\cap\K^\perp = \{0\}$.  If not, then $\dim (\widehat K\cap\K^\perp)$ would be nonzero and \eqref{mr2.5} would yield
\begin{align*}
\dim \X &=  \dim\widehat K+ \dim \K^\perp -\dim \widehat K\cap\K^\perp\\
&<  \dim\K+ \dim \K^\perp =\dim \X,
\end{align*}
a contradiction.  
Thus, we obtain that $\widehat K\cap\K^\perp = \{0\}$ which, together with \eqref{mr2.5}, implies that $\dim \widehat K = \dim \K$ and that the restriction of $\nabla\psi(\ox)+A$ to $\K$, denoted $(\nabla\psi(\ox)+A)\vert_\K$, is an one-to-one linear mapping from $\K$ to $\widehat K$.
Set $H:=P_\K\circ(\nabla\psi(\ox)+A)\vert_\K$.  Clearly, $H$ is also an one-to-one linear operator from $\K$ into $\K$ itself.
Turning to the proof  the implication in \eqref{mr9} for $G$, suppose that $0\in \widetilde D G(\ox, 0)(w) = DG(\ox, 0)(w)$. We then get from \eqref{mr2.3} that $w\in \K$ and $-(\nabla\psi(\ox)+A) (w)\in \K^\bot$.
 These indicate that 
\begin{equation*}
(\nabla\psi(\ox)+A) (w) \in \widehat K\cap \K^\bot =  \{0\},
\end{equation*} 
which in turn leads us to   $H( w)=P_\K(0)=0$.
Since $H$ is a bijection,  we get $w=0$, proving  \eqref{mr9} for $G$, which confirms that $G$ is strongly metrically regular at $\ox$ for $0$. 
This amounts to saying that  $S = G^{-1}$ has a Lipschitz continuous single-valued localization  around $0$ for $\ox$.
So, we  find neighborhoods $U$ of $\ox$ and $V$ of $0$ such that the mapping $y\mapsto S(y)\cap U$ is single-valued and Lipschitz continuous on $V$. Define the mapping $s:V\to U$ by $s(y)=S(y)\cap U$ for  $y\in V$. We then conclude that 
\begin{equation}\label{local}
\gph s = \gph S\cap (V\times U),
\end{equation}
 which clearly yields $T_{\gph s}(0, \ox) = T_{\gph S}(0, \ox)$. Recall that $G$ is strictly proto-differentiable at $\ox$ for $0$. It is worth mentioning that strict proto-differentiability is a geometric property of the graph of the mapping and that  $\gph G^{-1}$ and $\gph s$ are the same around $(0,\ox)$. 
Thus, strict proto-differentiability of $G^{-1}$   at $0$ for $\ox$ yields that of    $s$  at $0$ for $\ox$. 
By Definition~\ref{sss}(b),  $\gph s$ is strictly smooth at $(0, \ox)$. Since $s$ is Lipschitz continuous on $V$, we deduce from \cite[Proposition~3.1]{r85} that $s$ is  strictly differentiable at $0$.
To justify the formula for $\nabla s(0)$ in \eqref{derS}, take  $u\in \X$. We conclude from differentibility of $s$ at $0$ that $DG^{-1}(0,\ox)(u)=Ds(0)(u) = \nabla s(0)(u)$, which in turn yields $u\in DG(\ox, 0)(\nabla s(0)(u))$. Appealing now to \eqref{mr2.3}, we obtain  
\begin{equation*}\label{mr2.4}
\nabla s (0)(u) \in \K \quad \textrm{ and }\quad u-(\nabla\psi(\ox) +A)(\nabla s(0)(u))\in \K^\bot.
\end{equation*} 
Projecting the left-hand side of the second relation and using the first one yield
$$
P_\K(u) = P_\K\big((\nabla\psi(\ox) +A)(\nabla s(0)(u))\big)=H(\nabla s(0)(u)).
$$
Recalling that $H:\K\to\K$ is one-to-one, we arrive at
\begin{equation*}
\nabla s(0)(u) = H^{-1}(P_\K(u))
=  \big(P_\K\circ(\nabla\psi(\ox) + A)\vert_\K\big)^{-1} (P_\K(u)).
\end{equation*}
This confirms  \eqref{derS}  and hence completes the proof of  the implication (a)$\implies$(b). To prove the opposite implication, assume that (b) is satisfied. 
Thus, $G$ is strongly metrically regular at $\ox$ for $0$, which clearly shows that it is metrically regular at $\ox$ for $0$. Moreover, 
we find  neighborhoods $U$ of $\ox$ and $V$ of $0$ for which \eqref{local} holds. Since $s$ is strictly differentiable and Lipschitz continuous around $0$, we 
deduce from \cite[Proposition~3.1]{r85} that $s$ is strictly proto-differentiable at $0$ for $\ox$. It follows   from  \eqref{local} that the solution mapping $S$   is strictly proto-differentiable at $0$ for $\ox$.
Since $S=G^{-1}$, we conclude that $G$   is strictly proto-differentiable at $\ox$ for $0$. Using the sum rule for strict proto-differentiability in  \cite[Proposition~5.3]{HJS22} and the fact that 
$G=\psi+\sub f$ tells us that $\sub f$ is strictly proto-differentiable at $\ox$ for $-\psi(\ox)$, which proves (a) and completes the proof.
\end{proof}

One question that may arise is how one can ensure continuous differentiability (${\cal C}^1$) of the localization $s$ in  Theorem~\ref{smrch}(b), a property that is useful in dealing with a number of applications including 
 the one discussed at the end of this section. A close look into the proof of Theorem~\ref{smrch} tells us that if we assume strict proto-differentiability of $\sub f$ at $x$ for $v$ for any pair $(x,v)\in \gph \sub f$
 close to $(\ox,-\psi(\ox))$ and if $\psi$ is $\C^1$ in a neighborhood of $\ox$, we can ensure strict proto-differentiability of  
  the mapping $G$ from \eqref{G} on a neighborhood of $(\ox,0)$.  This implies that the localization $s$ is strictly differentiable in a neighborhood of $0$, a property equivalent 
 to $\C^1$-smoothness of $s$ around $0$ according to \cite[Exercise~1D.8]{DoR14}. This brings us to the following observation. 
 
  \begin{Theorem} \label{c1fors}Assume that $\ox$ is a solution to the generalized equation in \eqref{ge} in which $\psi$ is ${\cal C}^1$ around  $\ox$ and  $f$ is   prox-regular and subdifferentially continuous at $\ox$ for $-\psi(\ox)$.
   Then the following properties are equivalent:
 \begin{enumerate}[noitemsep,topsep=2pt]
\item the mapping $G$, taken from \eqref{G}, is metrically regular at $\ox$ for $0$ and  $\partial f$ is strictly proto-differentiable at   $x$ for  $v$ for all  $(x,v)\in \gph \sub f$ close to $(\ox, -\psi(\ox))$;
\item the solution mapping $S$ from \eqref{sol} has a Lipschitz continuous single-valued localization $s$ around $0\in \X$ for $\ox$, which  is ${\cal C}^1$ around $0$. 
\end{enumerate} 
 \end{Theorem}

As an immediate consequence of Theorem~\ref{smrch}, we arrive at the following equivalence of metric regularity and strong metric regularity of $G$ in the presence of  a strict proto-differetiability assumption.

\begin{Corollary}[equivalence between metric regularity and strong metric regularity]\label{thm:MR1}  
Assume that $\ox$ is a solution to the generalized equation in \eqref{ge} in which $\psi$ is strictly differentiable at $\ox$ and  $f$ is   prox-regular and subdifferentially continuous at $\ox$ for $\ov:=-\psi(\ox)$ and 
that $\partial f$ is strictly proto-differentiable at the point $\ox$ for  $\ov$. Then,   the mapping $G$, taken from \eqref{G}, is metrically regular at $\ox$ for $0$ if and only if it is strongly metrically regular at $\ox$ for $0$.
\end{Corollary}
\begin{proof}
This directly falls out of the established equivalence in Theorem~\ref{smrch}. 
\end{proof}

As mentioned above, Dontchev and Rockafellar  in \cite[Theorem~1]{dr96} derived the equivalence of metric regularity and strong metric regularity of the solution mapping $S$ from \eqref{sol} when $f=\dd_C$, where $C$ is a polyhedral convex subset of $\X$.
In such a framework,  Corollary~\ref{thm:MR1} covers the aforementioned seminal result 
under the extra assumption of strict proto-differentiability of $N_C$ at $\ox$ for $-\psi(\ox)$. Note that it was shown in \cite[Theorem~3.2(c)]{HJS22} that the latter condition on $N_C$ amounts to the relative 
interior condition $-\psi(\ox)\in \ri N_C(\ox)$.  While not knowing  whether or not a similar conclusion can be achieved for prox-regular functions in general, we will demonstrate in the next section that 
such an observation holds true for   $\C^2$-decomposable functions. 
We should add to this discussion that our proof is fundamentally different from the approach exploited in \cite{dr96} which  relied heavily on Robinson's results in \cite{rob92}
 and did not utilize the theory of second-order variational analysis, used in this paper.  Note also that when $f$ in \eqref{ge} enjoys a certain composite representation, Corollary~\ref{thm:MR1}  boils down to our recent result in \cite[Theorem~4.3]{HaS22}.
 Finally, the recent result in \cite[Corollary~7.3]{go} presents a characterization of strong metric regularity of  graphically Lipschitzian mappings (cf. \cite[Definition~9.66]{rw}).
It is not clear, however,  whether the latter graphically Lipschitzian assumption does hold for the mapping $G$ from \eqref{G} and hence whether the latter result can be exploited in our setting.   
 
 It was conjectured in \cite[Conjecture~4.7]{dmn} that if the proper function $f:\X\to \oR$ is prox-regular and subdifferentially continuous and $\ox$ is a local minimum of $f$, then 
 metric regularity and strong metric regularity of $\sub f$ at $\ox$ for $0$ are equivalent. First, observe that the condition of $\ox$ being a local minimizer is not essential, since
 one can choose a ${\cal C}^2$ mapping $g:\X\to \R$ and a polyhedral convex set $C\subset \X$ and set $f=g+\dd_C$. Employing \cite[Theorem~1]{dr96} tells us that 
  the desired equivalence of   metric regularity and strong metric regularity of $\sub f$ at $\ox$ for $0$  holds without demanding that $\ox$ be a local minimum of $f$.
 Dropping the latter condition from \cite[Conjecture~4.7]{dmn}, one can find in \cite[Example~BE.4]{kk} an example of a ${\cal C}^{1}$ function with Lipschitz continuous derivative (thus prox-regular
 and subdifferentially continuous by \cite[Proposition~13.34]{rw})  that is metrically regular
 but is not strongly metrically regular. This example suggests that  the latter equivalence does not hold for prox-regular in general. 
Corollary~\ref{thm:MR1} provides an answer to this question by demonstrating that 
if, in addition,  $\sub f$ is strictly proto-differentiable at $\ox$ for $\ov \in \partial f(\ox)$, we can ensure that metric regularity and strong metric regularity of $\sub f$ at $\ox$ for $\ov$  are equivalent, which confirms the  conjecture in \cite[Conjecture~4.7]{dmn} 
 under this extra assumption. While Dontchev and Rockafellar's result    in   \cite[Theorem~1]{dr96} indicates that strict proto-differentiability is not required for such an equivalence in general,
 it remains as an open question  to proceed if the strictly proto-differentiability condition fails. 
 
 When $\nabla \psi$ in the generalized equation in \eqref{ge} enjoys a certain symmetry property,  we can improve Corollary~\ref{thm:MR1}. To do this, recall that a set-valued mapping $F:\X\tto \Y$ is called strongly 
 metrically subregular at $x$ for  $y \in  F(x)$ if there exist a constant $\ell \ge 0$
and a neighborhood $U$ of $x$ such that the estimate $\|x'-x\|\le \ell\, \dist (y,F(x'))$ holds for any $x'\in U$.  

\begin{Corollary}\label{thm:MR4}  
Assume that $\ox$ is a solution to the generalized equation in \eqref{ge} in which $\nabla \psi(\ox)=\nabla \psi(\ox)^*$  and  $f$ is   prox-regular and subdifferentially continuous at $\ox$ for $\ov:=-\nabla \psi(\ox)$ and 
that $\partial f$ is strictly proto-differentiable at the point $\ox$ for  $\ov$. Then   
the following properties are equivalent:
 \begin{enumerate}[noitemsep,topsep=2pt]
 \item  the mapping $G$, taken from \eqref{G}, is strongly metrically regular  at $\ox$ for $0$;
\item the mapping $G$ is metrically regular at $\ox$ for $0$;
\item  the mapping $G$ is strongly metrically subregular  at $\ox$ for $0$.
\end{enumerate}
\end{Corollary}
\begin{proof}
 The equivalence of (a) and (b) was already established in Corollary~\ref{thm:MR1}. To prove (b) and (c) are also equivalent, we deduce from  \cite[Theorem~4E.1]{DoR14}  that $G$  is strongly metrically subregular  at $\ox$ for $0$
 if and only if the implication 
 $$
 0\in   D G(\ox, 0)(w)=\nabla \psi(\ox)w+ D(\sub f)(\ox, \ov)(w)\; \Longrightarrow\; w=0
 $$
 holds, where the equality comes from \eqref{mr2.3}.  By  \cite[Theorem~3.3(ii)]{Mor18},  $G$ is metrically regular at $\ox$ for $0$ if and only if  $0\in D^*G(\ox, 0)(w)$ yields $w=0$. 
 By the first equality in \eqref{cosum} and the assumption $\nabla \psi(\ox)=\nabla \psi(\ox)^*$, we conclude from   \eqref{gdcod} that 
 $$
 D^*G(\ox, 0)(w)=\nabla \psi(\ox)^*w+ D^*(\sub f)(\ox, \ov)(w)=\nabla \psi(\ox)w+ D(\sub f)(\ox, \ov)(w),
 $$
and hence we arrive at   $D G(\ox, 0)(w)=D^*G(\ox, 0)(w)$ for any $w\in \X$. This, coupled with the above characterizations of strong metric subregularity and metric regularity, demonstrates that 
 (b) and (c) are also equivalent and so completes the proof.
 \end{proof}
 
 Note that the condition $\nabla \psi(\ox)=\nabla \psi(\ox)^*$ in  the generalized equation in \eqref{ge} is satisfied for  an important instance  of generalized equations, namely the KKT system of optimization problems;
 see Theorem~\ref{mrkkt} for more detail. 
 

 We proceed with a characterization of  continuous differentiability   of the proximal mapping of prox-regular functions. Recall that proximal mapping 
of a function $f:\X\to \oR$ for a parameter value  $\gamma>0$,  denoted by  $\prox_{\gamma f}$, is  defined  by 
\begin{equation*}\label{proxmap}
\prox_{\gamma f}(x)= \argmin_{w\in \X}\Big\{f(w)+\frac{1}{2\gamma}\|w-x\|^2\Big\}, \quad x\in \X.
\end{equation*}
Recall from  \cite[Exercise~1.24]{rw}  that a function $f:\X\to \oR$ is called  prox-bounded if the function $f+\al\|\cdot\|^2$ is bounded from below on $\X$ for some $\al\in \R$. 
We begin with recording   some properties of the proximal mapping of prox-regular functions from \cite[Proposition~13.37]{rw}, which will be used in our characterization of continuous differentiability of their proximal mappings. 

\begin{Proposition}\label{proj1}
 Assume that  $f:\X\to \oR$ is   prox-regular  and subdifferentially continuous at $\ox$  for  $\ov\in \sub f(\ox)$ and that $f$ is prox-bounded. 
 Then there exist positive constants $\ve$ and $r$ such that for any $\gamma\in (0,1/r)$, there is a neighborhood $U_\gamma$ of $\ox+\gamma\ov$ on which $\prox_{\gamma\ph}$ is nonempty, single-valued,   Lipschitz continuous, and 
 can be calculated by 
\begin{equation}\label{proj3}
\prox_{\gamma f}=(I+\gamma T_\ve)^{-1},
\end{equation}
where  the set-valued mapping $T_\ve:\X\tto \X$ is defined by 
\begin{equation}\label{logh}
T_\ve(x)=\begin{cases}
\sub f(x)\cap \B_{\ve}(\ov)&\mbox{if}\;\; x\in \B_\ve(\ox),\\
\emptyset& \mbox{otherwise},
\end{cases}
\end{equation}
and where  $I$ stands for the identity operator on $\X$. Moreover, we have 
\begin{equation}\label{mepr}
\nabla e_\gamma f(x)=\frac{1}{\gamma}(x-\prox_{\gamma f}(x)), \;\; x\in U_\gamma,
\end{equation}
where the Moreau envelope function   $e_{\gamma} f$ is defined by 
$$
e_\gamma f(x)=\inf_{w\in \in \X}\big\{f(w)+\frac{1}{2\gamma}\|w-x\|^2\big\}, \;\; x\in \X.
$$
In addition, if $f$ is convex, then $T_\ve$ in \eqref{proj3} can be replaced with $\sub f$ and the  constant $r$ can be taken as $0$ with convention $1/0=\infty$.
\end{Proposition}

\begin{Theorem}\label{chpr}
Assume that  $f:\X\to \oR$ is   prox-regular and subdifferentially continuous  at $\ox$  for  $\ov \in \sub f(\ox)$ and that $f$ is prox-bounded. 
Then the following properties are equivalent:
  \begin{itemize}[noitemsep,topsep=2pt]
\item [ \rm {(a)}]  there exists a positive constant   $r$ such that for any $\gamma\in (0,1/r)$, the proximal mapping $\prox_{\gamma f}$ is ${\cal C}^1$ around $\ox+\gamma\ov$;
\item [ \rm {(b)}]  there exists a positive constant   $r$ such that for any $\gamma\in (0,1/r)$, the envelope function   $e_{\gamma} f$ is ${\cal C}^2$ around $\ox+\gamma\ov$;
\item [ \rm {(c)}]  the subgradient mapping  $\partial f$ is strictly proto-differentiable at   $x$ for  $v$ for all  $(x,v)\in \gph \sub f$ close to  $(\ox, \ov)$. 
\end{itemize}
If, in addition, $f$ is convex, then the constant $r$ in {\rm(}a{\rm)} and {\rm(}b{\rm)}   can be taken as $0$ with convention $1/0=\infty$.
\end{Theorem}
\begin{proof} 
The equivalence of (a) and (b) results directly from \eqref{mepr}.
We now show that (a) is also equivalent to (c). To do so,
take the positive constants $\ve$ and $r$ from Proposition~\ref{proj1}, pick  $\gamma\in (0,1/r)$, and take also the neighborhood $U_\gamma$ of $\ox +\gamma \ov$ from Proposition~\ref{proj1}.
Define the mapping $\psi:\X\to \X$ by $\psi(x)=x-(\ox+\gamma\ov)$ for any $x\in \X$
and consider the generalized equation 
\begin{equation}\label{gen}
0\in \psi(x)+\gamma T_\ve(x),
\end{equation}
where $T_\ve$ comes from \eqref{logh}. It is not hard to see that $\ox$ is a solution to \eqref{gen}. Define now 
the solution mapping  $S:\X\tto \X$ by 
$$
S(u):=\big\{x\in \X|\, u\in \psi (x)+ \gamma T_\ve(x)\big\},\quad u\in \X. 
$$
 Since $U_\gamma$ is a neighborhood of $\ox+\gamma\ov$, we find a $\dd>0$ such that $ \dd\B\subset U_\gamma -(\ox+\gamma\ov)$.
Take $u\in \dd\B$ and observe that $x\in S(u)$  amounts to $u+\ox+\gamma\ov\in (I+\gamma T_\ve)(x)$, which is equivalent via \eqref{proj3} to
$$
x= (I+\gamma T_\ve)^{-1}(u+\ox+\gamma\ov)=\prox_{\gamma f}(u+\ox+\gamma\ov).
$$
 These allow us to conclude that  $S(u)=\prox_{\gamma f}(u+\ox+r\ov)$ whenever $u\in \dd\B$.
Suppose that  (a) holds. We deduce from the latter that $S$ is ${\cal C}^1$ around $0$. By Theorem~\ref{c1fors}, the mapping $ \gamma T_\ve$ is strictly proto-differentiable at  $x$ for  $\gamma v$
   for all  $(x,v)\in \gph \sub f$ close to $(\ox, \ov)$. Since $\gph T_\ve=(\gph \sub f)\cap \big(\B_\ve(\ox)\times \B_\ve(\ov)\big)$, the latter property of $\gamma T_\ve$  is equivalent to (c) and hence we are done with (a)$\implies$(c).
   
Assume now that (c) is satisfied. Thus, $\gamma T_\ve$ is strictly proto-differentiable at  $x$ for  $\gamma v$ for all  $(x,v)\in \gph \sub f$ close to  $(\ox, \ov)$.
By  \cite[Proposition~5.3]{HJS22}, $\psi+\gamma T_\ve$ is  is strictly proto-differentiable at  $x$ for  $x+\gamma v -(\ox+\gamma \ov)$ for all  $(x,v)\in \gph \sub f$ close to  $(\ox, \ov)$. Using again $S(u)=\prox_{\gamma f}(u+\ox+r\ov)$ for all $u\in \dd\B$
and employing Proposition~\ref{proj1}, we get that $S$ is single-valued and Lipschitz continuous on $\dd\B$. Appealing now to \cite[Propostion~3G.1]{DoR14} and using the facts that $S(0)=\ox$ and $S=(I+\gamma T_\ve)^{-1}$ particularly tell us that 
$ \psi + \gamma T_\ve$ is metrically regular at $\ox$ for $0$.
Note that we can apply Theorem~\ref{smrch}, and therefore also Theorem~\ref{c1fors}, to the generalized equation in \eqref{gen} since 
$\gph T_\ve=(\gph \sub f)\cap \big(\B_\ve(\ox)\times \B_\ve(\ov)\big)$ due to  \eqref{logh} and since Theorem~\ref{c1fors} exploits only local points of $\gph \sub f$ close to $(\ox, -\psi(\ox))$ in the generalized equation in \eqref{ge}. 
It follows from  Theorem~\ref{c1fors} that the solution mapping $S$ has a Lipschitz continuous localization around $0$ for $\ox$, which is ${\cal C}^1$ in a neighborhood of $0$. Since $S(u)=\prox_{\gamma f}(u+\ox+\gamma \ov)$ for $u\in \dd\B$, we conclude that the proximal mapping $\prox_{\gamma f}$ is ${\cal C}^1$ in a neighborhood of $\ox+\gamma \ov$,
which proves  the implication (c)$\implies$(a).

If $f$ is convex, the same argument can be utilized to justify the  equivalence   (a)-(c) using the fact that  in this case, $T_\ve$ can be replaced with $\sub f$; see the final part of  Proposition~\ref{proj1}. This completes the proof.
\end{proof}

The characterization of continuous differentiability of proximal mappings using a strict proto-differentiability assumption on the subgradient mapping as part (c) in  Theorem~\ref{chpr} was first developed by Poliquin and Rockafellar in \cite[Theorem~4.4]{pr2}.
The latter result, however, requires in the setting of Theorem~\ref{chpr} that $\ox$ be a global minimizer of $f$, a condition that was replaced in our result by prox-boundedness of $f$, which is more realistic. We should emphasis that 
it is very likely that the latter requirement on $\ox$ in \cite[Theorem~4.4]{pr2} can be dropped by inspecting carefully its proof. We, however, didn't proceed in that way and justify this result as an immediate consequence of the equivalence of
metric regularity and strong metric regularity for generalized equations under strict proto-differentiability in Theorem~\ref{smrch}. Note that while Theorem~\ref{chpr} presents not only a sufficient condition but a characterization of 
continuous differentiability of proximal mappings of prox-regular functions, it still requires more effort to disentangle the strict proto-differentiability assumption in part (c) therein. This was accomplished for polyhedral functions in \cite[Theorem~3.5]{HJS22} and a certain 
composite functions in \cite[Theorem~3.10]{HaS23} by showing that the latter strict proto-differentiability assumption amounts to a relative interior condition. While we will be pursuing a similar characterization for 
$\C^2$-decomposable functions in the next section, it remains as an open question whether a similar result can be justified for a prox-regular function in general. 

Note that continuous differentiability of the projection mapping to convex sets was studied by Holmes in \cite{hol} 
in Hilbert spaces. His main result, \cite[Theorem~2]{hol}, states 
that if $\Omega\subset \R^d$ is a closed convex set, $x\in \R^d$,  the boundary of  $\Omega$ is a ${\cal C}^2$ smooth manifold  around $y=P_\Omega(x)$, then the projection mapping $P_\Omega$ is ${\cal C}^1$ in a neighborhood of the open normal ray 
$\{y+t(x-y)|\; t>0\}$; see also \cite[Theorem~2.4]{cst} for an extension of Holmes' result for prox-regular sets.

\section{Chain Rule for Strict  Proto-Differentiability}\label{chain}
The final section of this paper is devoted to provide a simple and verifiable  characterization of strict proto-differentiability for a class of composite functions, which encompasses important examples of functions that we often encounter 
in different classes of constrained and composite optimization problems. To achieve this goal, 
 recall from \cite{sh03} that $g: \Y \to \oR$ is said to be $\C^2$-decomposable at $u \in \Y$ if $g(u)$ is finite and $g$ can be locally represented in the composite form
\begin{equation}\label{co}
g(u') = g(u) + \vartheta(\Xi(u'))\quad \textrm{ for }\quad u'\in \O,
\end{equation}
where $\O \subset \Y$ is an open neighborhood of the given point $u$, $\vartheta: \Z \to \oR$ is a proper, lsc, sublinear function, and $\Xi:\O\to \Z$ is a $\C^2$-smooth mapping with $\Xi(u) = 0$  and $\Z$ being a finite dimensional Hilbert space.
One can immediately conclude  from \cite[Definition~3.18 and Exercise~3.19]{rw} that $\vartheta(\Xi(u)) = 0$.  
Variational analysis of $\C^2$-decomposable functions often requires a constraint qualification. The most common condition, used for this purpose, is called the {\em nondegeneracy} condition.
 In what follows, we say that the nondegeneracy condition  is satisfied for a $\C^2$-decomposable function $g$ with representation in \eqref{co} at $u\in \Y$
 if 
\begin{equation}\label{nondeg}
\para \{\partial \vartheta(\Xi(u))\} \cap \ker \nabla \Xi(u)^* = \{0\}
\end{equation}
holds, where $\para \{\partial \vartheta(\Xi(u))\} $  stands for the linear subspace of $\Z$ parallel to the affine hull of $ \sub \vartheta(\Xi(u))$. 
It is important to note that the nondegeneracy condition in \eqref{nondeg} holds automatically for many important examples of $\C^2$-decomposable functions; see \cite[Examples~2.1 and 2.3]{sh03}.
Since our analysis in this section  heavily relies on the nondegeneracy condition in \eqref{nondeg}, we call the function $g$ {\em reliably} $\C^2$-decomposable at $u \in \Y$ if both conditions \eqref{co} and \eqref{nondeg}
are satisfied concurrently. 

As shown in \cite[Example~2.4]{sh03}, the class of ${\cal C}^2$-{decomposable} functions
is a generalization of ${\cal C}^2$-{\em cone reducible} sets in the sense of \cite[Definition~3.135]{bs}, which is defined as follows:
A closed convex set $C\subset \Y$ is ${\cal C}^2$-{cone reducible} at $u\in C$ to a closed convex cone $\Th\subset \Z$ if there exist a neighborhood $\O\subset \Y$ of $u$ and a $\C^2$-smooth mapping $\Xi: \Y \to\Z$ such that 
\begin{equation}\label{codef}
C\cap\O =\big\{ u'\in \O\, \big|\, \Xi(u') \in \Theta\big\}, \quad \Xi(u) = 0, \; \textrm{ and }\quad \nabla \Xi(u): \Y \to \Z\; \textrm{ is surjective}.
\end{equation}
Below, we provide some important examples of reliably ${\cal C}^2$-{decomposable} functions that often appear in constrained and composite optimization problems.
\begin{Example} \label{c2rex} Assume that $g: \Y \to \oR$ and $\ou\in \Y$ with $g(\ou)$ finite. 
\begin{enumerate}[noitemsep,topsep=0pt]
\item If $C$ is ${\cal C}^2$-{cone reducible} at $\ou\in C$, then $g=\dd_C$ is reliably ${\cal C}^2$-{decomposable} at $\ou$.
 It is known that polyhedral convex sets (cf. \cite[Example~3.139]{bs}), the second-order cone, and the cone of $n\times n$ positive semidefinite matrices, denoted by $\S^n_+$,  (cf. \cite[Example~3.140]{bs}) are ${\cal C}^2$-{cone reducible} at all of their points.
 
 \item If $g$ is a polyhedral function,    it was shown in \cite[Example~2.1]{sh03}  that $g$ is reliably ${\cal C}^2$-{decomposable} at any points of its domain. 
 Furthermore, if $\Phi:\X\to \Y$ is a ${\cal C}^2$-smooth mapping, $\ou:=\Phi(\ox)\in \dom g$ for some $\ox\in \X$,  and the condition 
\begin{equation}\label{socq}
\para \{\partial g(\Phi(\ox))\} \cap \ker \nabla \Phi(\ox)^* = \{0\}
\end{equation}
is satisfied, the composite function $f=g\circ \Phi$ is reliably ${\cal C}^2$-{decomposable} at $\ox$. To justify it, it follows from 
\cite[Proposition~2.1(b)]{HJS22} that there is a neighborhood $\O$ of $\ou$ for which we have 
\begin{equation}\label{c2r}
g(u)=g(\ou)+\vartheta(u-\ou)\quad \mbox{with}\;\; \vartheta:=\d g(\ou),
\end{equation}
which leads us to the representation 
$$
f(x)=f(\ox)+\vartheta(\Phi(x)-\ou)
$$
for all $x$ close to $\ox$. Moreover, we deduce from  \cite[Exercise~8.44(b)]{rw} that  $\sub \vartheta(0)\subset \sub g(\ou)$. It is not hard to see via \eqref{c2r} that 
the latter inclusion becomes equality, and thus we obtain $\para \{\partial g(\ou)\} =\para \{\partial \vartheta(0)\} $. 
Combining this with \eqref{socq} confirms  that  the nondegeneracy condition in \eqref{nondeg} is satisfied at $u=\ou$, and hence  proves that $f$ is reliably ${\cal C}^2$-{decomposable} at $\ox$. 

\item 
Given $i\in \{1,\ldots,n\}$ and $X\in \S^n$, denote by $\ell_i(X)$
the number of eigenvalues that are equal to $\lm_i(X)$ but are ranked before $i$ including   $\lm_i(X)$. 
In what follows,   we often drop $X$ from  $\ell_i (X)$ when the dependence of $\l_i$ on $X $ can be seen clearly from the context. 
 This integer allows us to locate $\lm_i$
in the group of the eigenvalues of  $X$ as follows:
$$
\lm_1(X)\ge \cdots \ge \lm_{i-\ell_i(X)}>\lm_{i-\ell_i+1}(X)= \cdots =\lm_i(X)\ge \cdots\ge \lm_n(X).
$$
The eigenvalue $\lm_{i-\ell_i(X)+1}(X)$, ranking first in the group of eigenvalues equal to $\lm_i(X)$,  is called a {\em leading}
eigenvalue of $X$. For any $i\in \{1,\ldots,n\}$, define now the function $\al_i:{\S}^n\to \R$ by 
\begin{equation*}\label{fal}
\al_i(X)= \lm_{i-\ell_i+1}(X)+ \cdots +\lm_i(X),\quad X\in \S^n.
\end{equation*}
According to \cite[Example~2.3]{sh03},  $\al_i$ is reliably ${\cal C}^2$-{decomposable} at any $X\in \S^n$. In particular, when $\lm_i(X)$ 
ranks first in a group of equal eigenvalues, meaning either $i=1$ or $\lm_{i-1}(X)>\lm_i(X)$ if $i>1$, 
the function $\al_i$ reduces to $\lm_i$. This tells us that all the leading eigenvalue functions  are 
 always reliably ${\cal C}^2$-{decomposable}. Note that, except the first leading eigenvalue function,  the other leading eigenvalue functions  are nonconvex functions, 
 hence they provide  examples of nonconvex functions satisfying the reliable ${\cal C}^2$-decomposability property.
Despite the latter fact, it was shown in \cite[Theorem~2.3]{t98} that all the leading eigenvalue functions are subdifferentially regular.

Given $i\in \{1,\ldots,n\}$, define the sum of the first $i$ largest eigenvalues of $X$ by   
\begin{equation*}\label{fsig}
g_i(X)= \lm_{1}(X)+ \cdots +\lm_i(X).
\end{equation*}
It follows from \cite[Exercise~2.54]{rw} that  $g_i$ is convex. Observe also that   $g_i(X)=\al_i(X)+g_{i-\ell_i}(X)$.
By \cite[Proposition~1.3]{t98},  $g_{i-\ell_i}$ is ${\cal C}^2$-smooth on ${\S}^n$. This, coupled with \cite[Remark~2.2]{sh03}, demonstrates that 
$g_i$ is  reliably ${\cal C}^2$-{decomposable} at any $X\in \S^n$.  The later can be  extended for singular values of a matrix; see  \cite[Example~5.3.18]{mi} for more details.  
The readers can find more examples of ${\cal C}^2$-{decomposable} functions in \cite[Section~5.3.3]{mi}.
\end{enumerate}

\end{Example}

Below, we record some properties of sublinear functions, which is often used in this section.
\begin{Proposition}\label{subl}
Assume that $\vartheta: \Z \to \oR$ is a proper, lsc, and sublinear function and $\oz=0$. Then the following properties are fulfilled.
\begin{enumerate}[noitemsep,topsep=0pt]
 
\item  For any $z\in \dom \vt$, we have
\begin{equation*}
\vartheta(z)=\sup\big\{\la \eta,z\ra |\, \eta\in \sub\vartheta(\oz)\big\}=:\sigma_{\partial \vartheta(\oz)}(z)\quad \mbox{and}\quad \sub \vartheta(z)=\argmax\big\{\la \eta,z\ra |\, \eta\in \sub\vartheta(\oz)\big\}.
\end{equation*}

\item For any $z\in \dom \vt$, the inclusion $\emptyset\neq\sub \vt(z)\subset \sub \vt(\oz)$ holds.

\item For any $w\in \dom \vt$, we have $\d \vartheta(\oz)(w)=\vartheta(w)$. 
\end{enumerate}
\end{Proposition} 
\begin{proof}
Part (a) follows  from  \cite[Theorem~8.24]{rw} and  \cite[Corollary~8.25]{rw}.
Recall also from \cite[Theorem~8.24]{rw} that $\cl(\dom \vt) = (\sub \vt(\oz)^\infty)^*$, where $\sub \vt(\oz)^\infty$ stands for the horizon cone of $\sub \vt(\oz)$ (cf. \cite[Definition~3.3]{rw}).
We then conclude that the probably unbounded  set $\sub\vt(\oz)$ in the maximization problem in part (a) can be replaced with a closed bounded convex subset provided that $z\in \dom \vt$.
Thus, a maximizer must exist and $\sub \vt (z)\neq \emptyset$ for all $z\in \dom \vt$.
 The claimed inclusion in (b) results from the the second equality in (a). Finally, (c) 
follows from the fact that $\oz=0$ and $\vartheta$ is positively homogeneous. 
\end{proof}

Suppose that $g: \Y\to \oR$ is  $\C^2$-decomposable at $u\in \Y$ with representation in \eqref{co}. 
 Given $(u, y) \in \gph \partial g$, we define the set  of Lagrange multipliers associated with  $(u,y)$  by
\begin{equation}\label{laset}
M(u,y):=\big\{ \mu\in\Z\;\big|\;   \nabla \Xi(u)^* \mu=y,\; \mu\in \sub \vartheta(\Xi(u))\big\}.
\end{equation}
The next result collects some simple consequences of the nondegeneracy condition in \eqref{nondeg}, important for  our analysis of reliably $\C^2$-decomposable  functions in this section.

\begin{Proposition}\label{prop:nondeg}
Assume that  $g: \Y\to \oR$ is reliably $\C^2$-decomposable at $\ou\in \Y$ satisfying the representation in \eqref{co} for $u=\ou$. Then, the following properties hold.
\begin{enumerate}[noitemsep,topsep=2pt]
\item There is a neighborhood $U\subset \O$ of $\ou$ such that for any  $u\in U\cap  \dom g$,  the nondegeneracy condition in \eqref{nondeg} and the following basic constraint qualification {\rm(}BCQ{\rm)}  
\begin{equation}\label{bcq}
N_{\dom \vartheta}(\Xi(u)) \cap \ker \nabla \Xi(u)^* = \{0\}
\end{equation}
hold. Moreover,  $g$ is   prox-regular and subdifferentially continuous at any $u\in U\cap  \dom g$ for any $y\in \sub g(u)$ and  is 
subdifferentially regular at any such $u$.

\item The Lagrange multiplier set $M(u, y)$ from \eqref{laset} is a singleton for any  $u\in U\cap  \dom g$ and any $y\in \sub g(u)$, where $U$ is taken from {\rm(}a{\rm)}. 
Moreover, the dual condition 
\begin{equation}\label{duq}
K_\vartheta(\Xi(u),  \mu)^*\cap\ker\nabla\Xi(u)^*=\{0\}
\end{equation}
and the equivalence
\begin{equation}\label{rint}
 y \in \ri \sub g(u)\; \Longleftrightarrow\; \mu \in \ri  \sub \vt(\Xi(u))
\end{equation}
are satisfied for any $u\in U\cap  \dom g$, where $\mu$ is a unique element in $M(u, y)$  for any $y\in \sub g(u)$.

\item The set-valued mapping $\widetilde\Xi:  \Y \tto  \Z$, defined by $\widetilde\Xi (u):=\Xi(u) + L^\perp$ with $L:= \para \{\partial \vartheta(\Xi(\ou))\}$, is metrically regular at $\ou$ for $\Xi(\ou)$.
\end{enumerate}
\end{Proposition}

\begin{proof}
To prove  (a), we begin by justifying the nondegeneracy condition in \eqref{nondeg} for any $u\in \dom g$ sufficiently close to $\ou$.
Observe that the condition \eqref{nondeg} at $u=\ou$ yields 
\begin{equation}\label{co8}
\para \{\partial \vartheta(\Xi(\ou))\} \cap \ker \nabla \Xi(u)^* = \{0\}
\end{equation} 
for all $u$ sufficiently close to $\ou$. 
Since $\Xi(\ou) = 0$, it follows from  Proposition~\ref{subl}(b) that  $\emptyset\neq\partial \vartheta(\Xi(u)) \subset \partial \vartheta(\Xi(\ou))$ for all $u$ with $\Xi(u)\in \dom\vt$, which then tells us that the nondegeneracy condition \eqref{nondeg} also holds for $u\in \dom g$ sufficiently close to $\ou$ such that \eqref{co8} is fulfilled. 
Thus, we find a neighborhood $U$ of $\ou$ such that for any 
 $u\in U\cap  \dom g$,  the nondegeneracy condition in \eqref{bcq} is satisfied. 
We are now going to show that the BCQ condition \eqref{bcq} holds at $u=\ou$. 
To do so, we claim that $N_{\dom \vartheta}(\Xi(\ou)) \subset \para \{\partial \vartheta(\Xi(\ou))\}$. 
Since $\Xi(\ou) = 0$,  it follows from  Proposition~\ref{subl}(c) that  $\d \vartheta(\Xi(\ou))(w) = \vartheta(w)$ for all $w\in \Z$, which in turn implies $K_\vartheta(\Xi(\ou), \mu) \subset \dom \vartheta$, where $\mu$ is taken arbitrarily from $\partial \vartheta(\Xi(\ou))$. 
Since  $\vartheta$ is sublinear,  $\dom \vartheta$ is a convex cone. Thus, we arrive at
\begin{equation}\label{crin}
N_{\dom \vartheta}(\Xi(\ou)) = (\dom \vartheta)^* \subset K_\vartheta(\Xi(\ou),  \mu)^* = T_{\partial \vartheta(\Xi(\ou))}( \mu) \subset \para \{\partial \vartheta(\Xi(\ou))\},
\end{equation}
which confirms our claim.  
We then conclude from \eqref{nondeg} for $u=\ou$ that the BCQ condition \eqref{bcq} holds at that point.
Shrinking the neighborhood $U$ of $\ou$, if necessary, and employing robustness of the normal cone mapping $N_{\dom \vartheta}$ and $\C^2$-smoothness of $\Xi$, we can argue further that the BCQ condition \eqref{bcq} is valid for all $u\in U\cap \dom g$.
Finally,  the composite function $g$ in \eqref{co}, satisfying the nondegeneracy condition in \eqref{nondeg},  is strongly amenable at any $u\in U\cap\dom g$  in the sense of \cite[Definition~10.23(a)]{rw} and therefore is prox-regular and subdifferentially continuous at all points $u\in U\cap  \dom g$ for any $y\in \sub g(u)$ according to \cite[Theorem~13.32]{rw}. The subdifferential regularity of $g$ at any such $u$ falls out of \cite[Exercise~10.25(a)]{rw},  which completes the proof of (a). 

Turning now to (b), take the neighborhood  $U$ from (a) and pick $u\in U \cap \dom g$ and $y\in \sub g(u)$.  It follows from \cite[Example~10.8]{rw}, the composite representation \eqref{co}, and \eqref{bcq} that
\begin{equation}\label{subg}
\sub g (u) = \nabla \Xi(u)^*\sub \vt(\Xi(u)).
\end{equation}
Since $y\in \sub g (u)$, the latter implies that $M(u, y)$ is nonempty.
We claim  that $M(u, y)$ is a singleton. 
Indeed, assuming that $ \mu^1,  \mu^2 \in M(u, y)$, we get from \eqref{laset} that $ \mu^1 -  \mu^2 \in \ker\nabla \Xi(u)^*$ 
and $ \mu^1 -  \mu^2 \in \para \{\partial \vartheta(\Xi(u))\}$. Employing the nondegeneracy condition in  \eqref{nondeg}, we arrive at  $ \mu^1 =  \mu^2$, 
which proves our claim. The dual condition in \eqref{duq} then follows from the inclusion in \eqref{crin} together with the validity of the nondegeneracy condition at 
any $u\in U\cap  \dom g$, which was established in (a). Regarding the equivalence in \eqref{rint}, by \cite[Proposition~2.44]{rw}, it is a consequence of \eqref{subg}.

Finally, to justify (c), define the set-valued mapping  $F:  \Y \tto \Z$ by $F(u)=L^\bot$. Observe that $\gph F = \Y \times L^\perp$.
Using the definition of coderivative from \eqref{coder}, we  obtain  $D^*F(\ou, 0)(w) = \{0\}$ for $w\in L$ and $D^*F(\ou, 0)(w) = \emptyset$ otherwise. 
Since $\widetilde\Xi (u)=\Xi(u) +F(u)$ for all $u\in \Y$, using  the sum rule for the coderivative from \cite[Exercise~10.43(b)]{rw} yields 
\begin{equation}\label{cod1}
D^*\widetilde\Xi(\ou, \Xi(\ou))(w) =
\nabla\Xi(\ou)^*w + D^*F(\ou, 0)(w)= \nabla\Xi(\ou)^*w\quad \mbox{for all}\;\; w \in L.
\end{equation}
By \cite[Theorem~9.43]{rw}, the mapping $\widetilde\Xi$ is metrically regular at $\ou$ for $\Xi(\ou)$ if and only if $0\in D^*\widetilde\Xi(\ou, \Xi(\ou))(w)$ yields $w=0$.
According to \eqref{cod1}, the condition $0\in D^*\widetilde\Xi(\ou, \Xi(\ou))(w)$ amounts to $w\in \para \{\partial \vartheta(\Xi(\ou))\} \cap \ker \nabla \Xi(\ou)^*$.
Since the nondegeneracy condition in \eqref{nondeg} holds at $u=\ou$, we obtain $w=0$, which proves (c) and hence completes the proof.
\end{proof}

We proceed with a chain rule for the second subderivative of reliably $\C^2$-decomposable functions.

\begin{Theorem} \label{prop:tedg}
Assume that  $g: \Y\to \oR$ is reliably $\C^2$-decomposable at $\ou\in \Y$ satisfying the  representation in  \eqref{co} for $u=\ou$.
Then,  there is a neighborhood $U$ of $\ou$ such that for any  $u\in U\cap  \dom g$ and any $y\in \sub g(u)$,   the second subderivative of $g$ at $u$ for $y$ can be calculated by
\begin{equation}\label{co3}
\d^2g (u, y)(w) = \la \mu, \nabla^2 \Xi(u)(w, w)\ra +\d^2\vartheta(\Xi(u), \mu)(\nabla\Xi(u)w), \quad \textrm{ for }\; w\in \Y,
\end{equation}
 where $\mu\in \Z$ is the unique element in the multiplier set $M(u, y)$ from \eqref{laset}.  
\end{Theorem}

\begin{proof} Take the neighborhood $U$ from Proposition~\ref{prop:nondeg}(a) and pick $u\in U\cap  \dom g$ and $y\in \sub g(u)$. 
By Proposition~\ref{prop:nondeg}(b), we know that the multiplier set $M(u, y)$ is a singleton, say $M(u,y)=\{\mu\}$. 
It follows from Proposition~\ref{prop:nondeg}(a) that the BCQ in \eqref{bcq} holds, which, coupled with   \cite[Theorem~13.14]{rw}, gives us  the inequality  
\begin{equation}\label{co4}
\d^2g (u, y)(w)\geq \la  \mu, \nabla^2 \Xi(u)(w, w)\ra +\d^2\vartheta(\Xi(u),  \mu)(\nabla \Xi(u)w)\quad \mbox{for all}\;\; w\in \Y.
\end{equation}
To obtain the opposite inequality, 
fix $w\in \Y$ and pick sequences $t_k\searrow 0$ and $\xi^k \to \nabla \Xi(u)w$ as $k\to \infty$, satisfying
\begin{equation}\label{co4.1}
\d^2\vartheta(\Xi(u),  \mu)(\nabla \Xi(u)w) = \lim_{k\to\infty}\frac{\vartheta(\Xi(u)+t_k\xi^k)-\vartheta(\Xi(u))-t_k\la  \mu, \xi^k\ra}{\frac{1}{2}t_k^2}.
\end{equation}
We know from the nondegeneracy condition in \eqref{nondeg} that  the set-valued mapping $\widetilde\Xi:  \Y \rightrightarrows \Z$, defined in Proposition~\ref{prop:nondeg}(c), 
is metrically regular at $\ou$ for $\Xi(\ou)$, meaning 
that there  exist $\ell\geq 0$ and  neighborhoods $V$ of $\ou$ and $W$ of $\Xi(\ou)$ for which 
the estimate 
$$
\dist(p,\widetilde\Xi^{-1}(q))\le \ell\, \dist(q, \widetilde\Xi(p))=  \ell\, \dist(q,  \Xi(p)+L^\perp)\quad \mbox{for all}\;\; (p,q)\in V\times W
$$
is satisfied, where $L=\para \{\partial \vartheta(\Xi(\ou))\}$.  Since $\Xi$ is ${\cal C}^2$-smooth, we can assume by shrinking the neighborhood $U$ if necessary that 
$(p^k,q^k):=(u+t_k w, \Xi(u)+t_k \xi^k)\in V\times W$ for any $k$ sufficiently large. 
Thus, for any $k$ sufficiently large, we can find $u^k\in \widetilde\Xi^{-1}(q^k)$  such that 
\begin{equation*}\label{co4.2}
\|u+t_k w - u^k \|=\|p^k-u^k \|\leq \ell\, \dist\big(q^k,\widetilde\Xi(p^k)) \leq \ell\, \| \Xi(u)+t_k \xi^k- \Xi(u+t_kw)\|.
\end{equation*}
Let $w^k := (u^k - u)/t_k$ and observe that
\begin{equation*}
\|w^k - w\| = \big\|\frac{u^k - u- t_k w}{t_k}\big\|\leq \ell \, \big\|\frac{\Xi(u+t_kw) - \Xi(u)}{t_k} - \xi^k\big\| \to 0 \quad\textrm{ as }\; k \to \infty,
\end{equation*}
which yields  $w^k\to w$ as $k\to \infty$.
Since $u^k\in \widetilde\Xi^{-1}(q^k)$, we find by the definition of $\widetilde\Xi$ a vector $\nu^k \in L^\perp$ such that 
\begin{equation}\label{xik}
\Xi(u) +t_k\xi^k =q^k= \Xi(u^k) +\nu^k.
\end{equation} 
Recall that $\mu\in M(u,y)$ and hence $\mu\in \sub \vartheta(\Xi(u))\subset \sub \vt(\Xi(\ou))$. 
Since $\nu^k \in L^\perp = \para\{\partial \vartheta(\Xi(\ou))\}^\perp$,
we get $\la \nu^k ,\eta\ra=\la \nu^k ,\mu\ra$ for any $\eta\in \sub \vartheta(\Xi(\ou))$. This, coupled with Proposition~\ref{subl}(a), allows us to conclude for all $k$ sufficiently large that
\begin{align}
\vartheta(\Xi(u) +t_k\xi^k) = \vartheta(\Xi(u^k)+\nu^k)&=\sigma_{\partial \vartheta(\Xi(\ou))} \big(\Xi(u^k)+\nu^k \big) =  \la \mu, \nu^k\ra + \sigma_{\partial \vartheta(\Xi(\ou))} \big( \Xi(u^k) \big) \nonumber\\
&= \la \mu, \nu^k\ra+  \vartheta(\Xi(u^k)). \label{co4.4}  
\end{align}
We therefore obtain from  \eqref{co}, \eqref{co4.1}, and \eqref{xik} as well as $w^k\to w$ that
\begin{align}
&\d^2\vartheta(\Xi(u),  \mu)(\nabla \Xi(u)w) = \lim_{k\to\infty}\frac{\vartheta(\Xi(u^k))+\la\mu, \nu^k\ra-\vartheta(\Xi(u))-\la  \mu, \Xi(u^k) +\nu^k -\Xi(u)\ra}{\frac{1}{2}t_k^2}\nonumber\\
&= \lim_{k\to\infty}\Big(\frac{g (u+t_kw^k)- g(u)-t_k\la \nabla \Xi(u)^* \mu, w^k\ra}{\frac{1}{2}t_k^2}-\frac{\la  \mu, \Xi(u^k)-\Xi(u) - t_k\nabla \Xi(u)w^k\ra}{\frac{1}{2}t_k^2}\Big)\nonumber\\
&\geq \d^2g (u, y)(w) -\la  \mu, \nabla^2\Xi(u)(w, w)\ra,\nonumber
\end{align}
which   verifies the opposite inequality in \eqref{co4}   and hence completes the proof. 
\end{proof}

Note that the chain rule for the second subderivative of reliably $\C^2$-decomposable functions, established above, cannot be obtained from  previous efforts in this direction.
In fact, it was shown in \cite[Theorem~5.4]{ms20} that such a chain rule can be achieved for a broader class of functions, which includes $\C^2$-decomposable functions, under the metric subregularity constraint qualification, 
which is strictly weaker than the nondegeneracy condition, imposed in Theorem~\ref{prop:tedg}. However, \cite[Theorem~5.4]{ms20} gives us this chain rule just at the point under consideration 
and cannot ensure a similar result for all the points nearby. Indeed, had we  assumed  in Theorem~\ref{prop:tedg} the reliable $\C^2$-decomposability of $g$ for all $u$ close to $\ou$, we could have applied \cite[Theorem~5.4]{ms20}
to obtain \eqref{co3}. Even in such a case, we should have dealt with the fact that the mapping $\Xi$ in \eqref{co3} would depend on $u$, an unpleasant issue which causes a major obstacle in the proof of Theorem~\ref{c2cone}.
The nondegeneracy condition allows us to bypass this hurdle by exerting  metric regularity   into the mapping $\widetilde\Xi$ in Propsoition~\ref{prop:nondeg}(c). Note also that a similar chain rule  was recently obtained 
in \cite[Corollary~4.3]{bm} for a general composite function under the surjectivity assumption on the jacobian of the inner mapping. Imposing the latter surjectivity condition in Theorem~\ref{prop:tedg}, one can obtain 
\eqref{co3} without facing the above-mentioned  obstacle. We, however, require to proceed with the nondegeneracy condition, which is in general strictly weaker than the surjectivity condition.

To establish our characterization of strict proto-differentiability of subgradient mappings of $\C^2$-decomposable functions, we 
 begin to analyze some  variational properties  of the  sublinear function $\vartheta$ from \eqref{co}  at the origin.

\begin{Lemma}\label{claim2}
Let $\vartheta: \Z \to \oR$ be a proper, lsc, and sublinear function, $\oz=0$, and $\bar\mu \in \ri \partial \vartheta(\oz)$. Then there exists $\varepsilon>0$ such that for all $(z,  \mu) \in \big(\gph \partial \vartheta\big)\cap \B_\varepsilon(\oz, \bar\mu)$, we have 
$
 \mu \in \ri \partial \vartheta(\oz)
$
and $\mu \in \ri \partial \vartheta(z)$. 
\end{Lemma}

\begin{proof} Since $\bar\mu \in \ri \partial \vartheta(\oz)$, we find  $\varepsilon>0$   such that $\big(\aff \partial \vartheta(\oz) \big)\cap \B_{2\varepsilon}(\bar\mu)\subset\partial \vartheta(\oz)$. 
Take any $(z,  \mu )\in \big(\gph \partial \vt\big) \cap \B_\varepsilon(\oz, \bar\mu)$. It is not hard to see that 
$$
  \big(\aff \partial \vartheta(\oz)\big)\cap \B_\varepsilon( \mu)\subset  \big(\aff \partial \vartheta(\oz) \big)\cap \B_{2\varepsilon}(\bar\mu)\subset \partial \vartheta(\oz),
  $$ 
  which in turn proves that $ \mu \in \ri \partial \vartheta(\oz)$. To prove the second claim, take  $\widehat q \in \big(\aff \partial \vartheta(z)\big)\cap \B_\varepsilon( \mu)$. 
We then deduce from $\mu\in \B_\ve(\omu)$ and Proposition~\ref{subl}(b) that   $\widehat q \in    \big(\aff \partial \vartheta(\oz)\big) \cap \B_{2\varepsilon}(\bar\mu)\subset\partial \vartheta(\oz)$.
Since  $\widehat q\in \aff \partial \vartheta(z)$, we find $m\in \N$, $\alpha_i \in \R$ and $q^i \in \partial \vartheta(z)$ for  $i=1,\ldots,m$   with $\sum_{i=1}^m \alpha_i = 1$ such that $\widehat q = \sum_{i=1}^m \alpha_i q^i$. 
We conclude from Proposition~\ref{subl}(a) and $q^i \in \partial \vartheta(z)$ 
 that $ \la q^i, z\ra=\beta$ for any $i=1,\ldots,m$, where $\beta:=\max_{q \in \partial \vartheta(\oz)}\,\la q, z\ra$. Thus, we obtain 
 \begin{equation*}
 \la \widehat q, z\ra =  \sum_{i=1}^m\alpha_i \la q^i, z\ra =\beta= \max_{q \in \partial \vartheta(\oz)}\,\la q, z\ra.
\end{equation*}
Employing again  Proposition~\ref{subl}(a) and using the fact that $\widehat q \in \partial \vartheta(\oz)$ imply that 
 $\widehat q \in \partial \vartheta(z)$. This proves the inclusion $\big(\aff \partial \vartheta(z)\big)\cap \B_\varepsilon( \mu) \subset \partial \vartheta(z)$, which yields $ \mu \in \ri \partial \vartheta(z)$ and  hence completes the proof.
\end{proof}

We continue with recording a characterization of strict proto-differentiability of subgradient mappings of prox-regular functions  via the concept of strict twice epi-differentiability.
Following \cite[page~1830]{pr},   we say that a function $f:\X\to \oR$ is {\em strictly} twice epi-differentiable at $\bar x$ for $\ov\in \sub f(\ox)$ if 
the functions $  \Delta_t^2 f(  x , v)$ epi-converge  to a function as $t\searrow 0$, $(x,v)\to (\ox,\ov)$ with $f(x)\to f(\ox)$ and $(x,v)\in \gph \sub f$.  If this condition holds, the limit function must be  the second subderivative $  \d^2 f(\bar x,\ov)$.

\begin{Proposition}\label{chste} Under the hypothesis of Corollary~{\rm\ref{sted}}, the following properties can be added to the list of equivalences:
\begin{itemize}[noitemsep,topsep=2pt]
\item [ \rm {(a)}] $  f$ is strictly twice epi-differentiable at $\ox$ for $\ov$;
\item [ \rm {(b)}] $\d^2 f(x,v)$ epi-converges to $\d^2  f(\ox,\ov)$  as $(x,v)\to (\ox,\ov)$  in the set of pairs  $(x,v)\in\gph \sub f$ for which $ f$ is   twice epi-differentiable;
\item [ \rm {(c)}] $\d^2 f(x,v)$ epi-converges to $\d^2  f(\ox,\ov)$  as $(x,v)\to (\ox,\ov)$  in the set of pairs  $(x,v)\in\gph \sub f$ for which $\d^2 f(x,v)$  is generalized quadratic. 
\end{itemize}
\end{Proposition}
\begin{proof} The equivalence of the properties in (a)-(c) to those in Corollary~{\rm\ref{sted}} were established in \cite[Corollary~4.3]{pr2} using the Attouch' theorem (cf. \cite[Theorem~12.35]{rw}).
Note that \cite[Corollary~4.3]{pr2} does not determine to which function $\d^2 f(x,v)$ epi-converges  in (b) and (c).  The properties in Corollary~{\rm\ref{sted}} allow us to demonstrate  that 
the latter function is indeed the second subderivative $\d^2  f(\ox,\ov)$. One can use a similar argument as \cite[Corollary~4.3]{pr2} to justify the claimed equivalences.
\end{proof}

Below, we present our first major result in this section, a characterization of strict twice epi-differentiability of sublinear functions.
\begin{Theorem}\label{septheta}
Let $\vartheta: \Z \to \oR$ be a proper, lsc, and sublinear function. Then,  $\vartheta $ is strictly twice epi-differentiable at $\oz = 0$ for $\bar\mu$ if and only if $\bar\mu \in \ri \partial \vartheta(\oz)$. 
\end{Theorem}

\begin{proof} Observe first from \cite[page~138]{mi} that 
\begin{equation}\label{co1}
\d^2\vartheta(\oz, \bar\mu)(w) = \delta_{K_\vartheta(\oz,\, \bar\mu)}(w), \quad \textrm{ for }\; w\in \Z.
\end{equation}
If $\vartheta $ is strictly twice epi-differentiable at $\oz = 0$ for $\bar\mu \in \partial \vartheta(\oz)$, it follows from 
  Proposition~\ref{chste} that $\partial \vartheta$ is strictly proto-differentiable at $\oz$ for $\bar\mu$. Moreover, by \eqref{co1}, we have $\dom \d^2\vartheta(\oz, \bar\mu)=K_\vartheta(\oz,\, \bar\mu)$.
Employing Proposition~\ref{pedcon} yields  $\bar\mu \in \ri \partial \vartheta(\oz)$.

Assume now that  $\bar\mu \in \ri \partial \vartheta(\oz)$. To justify strict twice epi-differentiability of $\vartheta$ at $\oz$ for $\bar\mu$, 
we claim that  the functions $\Delta^2_t \vartheta  (z, \mu)$ epi-converge  to $ \delta_{K_\vartheta(\oz,\, \bar\mu)}$ as 
$t\searrow 0$ and  $(z,\mu)\to (\oz,\bar\mu)$   with $(z,\mu)\in \gph \sub \vartheta$.  
According to \cite[Proposition~7.2]{rw},  the latter epi-convergence claim holds  if and only if for any $w \in \Z$,  any sequence $t_k\searrow 0$, and any sequence $(\zk, \mu^k) \to (\oz, \bar\mu)$ with $\mu^k\in \partial \vartheta(\zk)$, the second-order difference quotients $\Delta^2_\tk \vartheta  (\zk, \mu^k)$ satisfy the  conditions
\begin{subequations}\label{co2}
\begin{align}
&\liminf\limits_{k\to\infty}\Delta^2_\tk \vartheta  (\zk, \mu^k)(w^k)\geq \delta_{K_\vartheta(\oz,\, \bar\mu)}(w)\quad\textrm{for every sequence }\; w^k\to w,\label{co2:a}\\
&\limsup\limits_{k\to\infty}\Delta^2_\tk \vartheta  (\zk, \mu^k)(w^k)\leq \delta_{K_\vartheta(\oz,\, \bar\mu)}(w)\quad\textrm{for some sequence }\; w^k\to w.\label{co2:b}
\end{align}
\end{subequations}
We split the verification of \eqref{co2:a} and \eqref{co2:b} into  the two possible cases: $w\in K_\vartheta(\oz, \bar\mu)$ and $w\notin K_\vartheta(\oz, \bar\mu)$. 
Assume first that $w \in K_\vartheta(\oz, \bar\mu)$ and choose arbitrary sequences  $t_k \searrow 0$, and $(\zk, \mu^k)\to (\oz, \bar\mu)$ with $\mu^k\in \partial \vartheta(\zk)$. 
We deduce from the convexity of $\vartheta $ and $\mu^k \in \partial \vartheta(\zk)$  that
\begin{equation*}
\Delta^2_\tk \vartheta  (\zk, \mu^k)(w^k) = \frac{\vartheta(\zk +t_kw^k)-\vartheta(\zk) - t_k\la \mu^k, w^k\ra}{\frac{1}{2}t_k^2}\geq 0,
\end{equation*}
which clearly justifies \eqref{co2:a} for any sequence $w^k \to w$, since $w \in K_\vartheta(\oz, \bar\mu)$. 
To prove \eqref{co2:b}, set $w^k = w$ for all $k\in \N$. 
Since  $\bar\mu\in \ri \partial \vartheta(\oz)$, the critical cone $K_\vartheta(\oz, \bar\mu) = N_{\partial \vartheta(\oz)}(\bar\mu)$ is a linear subspace, where the last equality comes from Proposition~\ref{proxcri}(b). 
By Proposition~\ref{subl}(b), the inclusion  $\partial \vartheta(\zk) \subset \partial \vartheta(\oz)$ holds for any $k$, which, coupled with $w\in N_{\partial \vartheta(\oz)}(\bar\mu)$, yields $\la\mu^k - \bar\mu, w\ra = 0$ for all $k \in \N$. 
This, together with  Proposition~\ref{subl}(c), leads us to $\la \mu^k, w\ra = \la \bar\mu, w\ra = \d \vartheta(\oz)(w) = \vartheta(w)$ for all $k$, and so we  arrive at
\begin{equation*}
\Delta^2_\tk \vartheta  (\zk, \mu^k)(w) = \frac{\vartheta(\zk +t_k w)-\vartheta(\zk) - t_k \vartheta(w)}{\frac{1}{2}t_k^2} \leq 0
\end{equation*}
due to the sublinearity of $\vartheta$. Passing to the limit clearly proves   \eqref{co2:b} in this case. 

Suppose now that $w \notin K_\vartheta(\oz, \bar\mu)$. 
Clearly,  \eqref{co2:b} is fulfilled for any sequence $w^k \to w$ with $w\notin K_\vartheta(\oz, \bar\mu)$. Pick  arbitrary sequences $w^k \to w$ and $(\zk, \mu^k) \to (\oz, \bar\mu)$ with $\mu^k \in \partial \vartheta(\zk)$. 
Recalling that $\omu \in \ri \sub \vt(\oz)$, we deduce from Lemma~\ref{claim2} that $\mu^k \in \ri \partial \vartheta(\oz)$ for all $k$ sufficiently large, and therefore arrive at $N_{\partial \vartheta(\oz)}(\mu^k) = N_{\partial \vartheta(\oz)}(\bar\mu) = K_\vartheta(\oz, \bar\mu)$ for such $k$. 
It follows from Proposition~\ref{subl}(a) that $\zk \in N_{\partial \vartheta(\oz)}(\mu^k)$. 
Combining these tells us that $z^k\in K_\vartheta(\oz, \bar\mu)$, and also $-\zk \in K_\vartheta(\oz, \bar\mu)$ for $k$ sufficiently large, since
$K_\vartheta(\oz, \bar\mu)$ is a linear subspace of $\Z$. Thus, we conclude from Proposition~\ref{subl}(c) that 
$$
\d\vt(\oz)(\zk)=\vartheta(\zk) = \la\bar\mu, \zk\ra = -\la \bar\mu, -\zk\ra = -\d \vartheta(\oz)(-\zk)=- \vartheta(-\zk),
$$
which, together with the sublinearity of $\vartheta$, yields
\begin{equation*}
\Delta^2_\tk \vartheta  (\zk, \mu^k)(w^k) = \frac{\vartheta(\zk +t_kw^k)+\vartheta(-\zk)-t_k\la \mu^k, w^k\ra}{\frac{1}{2}t_k^2}\geq \frac{\vartheta(w^k) - \la \mu^k, w^k\ra}{\tfrac{1}{2}t_k}
\end{equation*}
for all $k$ sufficiently large. Observe that 
\begin{equation*}
\liminf_{k\to \infty} \vartheta(w^k) - \la \mu^k, w^k\ra \geq \vartheta(w) - \la \bar\mu, w\ra >0,
\end{equation*}
where the first inequality results from  lower semicontinuiuty of $\vartheta$ and the second one comes from  $w\notin K_\vartheta(\oz, \bar\mu)$.
Combining these estimates  and passing to the limit clearly justify \eqref{co2:a} for any $w \notin K_\vartheta(\oz, \bar\mu)$ and hence completes the proof. 
\end{proof}

We are now in a position to  establish our main result in the section, a simple characterization of strict proto-differentiability of subgradient mappings of 
$\C^2$-decomposable functions.

\begin{Theorem}[strict  proto-differentiability of $\C^2$-decomposable functions] \label{c2cone}
Assume that $g: \Y\to \oR$, $(\ou,\oy) \in \gph  \partial g$, and that there is a neighborhood $U$ of $\ou$ such that  $g$ is reliably $\C^2$-decomposable at any $u\in U$. 
Then the following properties are equivalent:
\begin{enumerate}[noitemsep,topsep=2pt]
\item $g $ is strictly twice epi-differentiable at $u$ for $y$ for any pair $(u,y)\in \gph \sub g$ close to $(\ou,\oy)$;  

\item $\partial g $ is strictly proto-differentiable at $u$ for $y$ for any pair $(u,y)\in \gph \sub g$ close to $(\ou,\oy)$;  

\item $\oy \in \ri \partial g (\ou)$.
\end{enumerate} 
\end{Theorem}

\begin{proof}
The equivalence of (a) and (b) results from Proposition~\ref{chste}.  We proceed by concluding from Theorem~\ref{prop:tedg} and \eqref{co1} that 
$$
\dom \d^2g (\ou, \oy)=\big\{w\in \Y\, \big|\, \nabla \Xi(\ou)w \in K_\vartheta(\Xi(\ou), \bar\mu)\big\},
$$
where $\bar\mu$ is the unique element of the multiplier set $M(\ou,\oy)$; see Proposition~\ref{prop:nondeg}(b). 
By Proposition~\ref{prop:nondeg}(a), the BCQ condition in \eqref{bcq} is satisfied at $u=\ou$. Thus, it follows from the chain rule for the subderivative in \cite[Theorem~10.6]{rw} that 
\begin{align*}
K_g(\ou,\oy) &= \big\{w\in \Y\, \big|\, \d g(\ou)(w) =  \d\vt(\Xi(\ou))(\nabla\Xi(\ou)w) = \la \oy, w\ra= \la \omu,  \nabla\Xi(\ou)w\ra\big\} \nonumber\\ 
&= \big\{w\in \Y\, \big|\, \nabla\Xi(\ou)w \in K_{\vartheta}(\Xi(\ou), \omu)\big\}. 
\end{align*}
Combining these tells us that $\dom \d^2g (\ou, \oy)=K_g(\ou,\oy)$. Moreover, $g$ is subdifferentially regular at $\ou$ due to Proposition~\ref{prop:nondeg}(a).
Thus,  Proposition~\ref{pedcon} confirms the implication (b)$\implies$(c).

We now turn to the implication (c)$\implies$(a). Assume first that $\oy \in \ri \partial g (\ou)$, $g$ has the representation in \eqref{co} with $u=\ou$, and the nondegeneracy condition \eqref{nondeg} holds at $u = \ou$. 
We are going to verify strict twice epi-differentiability of $g$ at $\ou$ for $\oy$ via justifying its characterization in  Proposition~\ref{chste}(b).
To this end, suppose that $(u, y)\to (\ou, \oy)$ with $y\in \partial g (u)$ and that $g $ is twice epi-differentiable at $u$ for $y$. 
The latter is equivalent via \cite[Proposition~7.2]{rw} to the fact  that for any $w\in \Y$ and any sequence of $t_k\searrow 0$,  there  exists a sequence $w^k \to w$ with
\begin{equation}\label{co6}
\lim_{k\to\infty} \frac{g (u+t_kw^k )-g(u) -t_k\la y, w^k \ra}{\frac{1}{2}t_k^2} = \d^2g (u, y)(w).
\end{equation}
Pick $(u, y)\in \gph \partial g$ sufficiently close to  $(\ou, \oy)$ such that $u+t_kw^k\in \O$  for any $k$ sufficiently large with $\O$ taken from \eqref{co}.  It   follows from \eqref{co} that 
$$
g (u+t_kw^k ) - g(u)=\vartheta(\Xi(u+t_kw^k )) - \vartheta(\Xi(u)).
$$ 
Set $\xi^k := (\Xi(u+t_kw^k )-\Xi(u))/t_k$. We  can rewrite the left-hand side of \eqref{co6} to obtain 
\begin{equation}\label{co6.1}
\d^2g (u, y)(w)=\lim_{k\to\infty} \Big (\Delta_{t_k}^2\vartheta(\Xi(u),  \mu)(\xi^k)+\frac{\la  \mu, \Xi(u+t_kw^k )-\Xi(u)-\nabla \Xi(u)w^k \ra}{\frac{1}{2}t_k^2}\Big),
\end{equation}
where $ \mu$ is the unique element in the multiplier set   $M(u, y)$; see Proposition~\ref{prop:nondeg}(b).
According to Theorem~\ref{prop:tedg},  \eqref{co3} holds for any $(u, y)\in \gph \partial g$ with $u$ sufficiently close to $\ou$.
This, coupled with \eqref{co6.1},  leads us to 
\begin{align}
\lim_{k\to\infty} \Delta_{t_k}^2\vartheta(\Xi(u),  \mu)(\xi^k) &=  \d^2g (u, y)(w)-\lim_{k\to\infty}  \frac{\la  \mu, \Xi(u+t_kw^k )-\Xi(u)-\nabla \Xi(u)w^k \ra}{\frac{1}{2}t_k^2}\nonumber\\
&= \d^2\vartheta(\Xi(u),  \mu)(\nabla \Xi(u)w).\label{co7}
\end{align}
Our goal is to show that $\vartheta$ is twice epi-differentiable at $\Xi(u)$ for $\mu$. In to order to achieve it, pick $\xi\in \Z$.
It is not hard to see that the nondegeneracy condition in \eqref{nondeg} for $u=\ou$ can be equivalently described as 
 $\Z = \rge \nabla \Xi(u) + \para\{\partial \vartheta(\Xi(\ou))\}^\perp$. Thus,  we can find $w\in \Y$ and $\nu \in \para\{\partial \vartheta(\Xi(\ou))\}^\perp$ such that $\xi= \nabla \Xi(u)w + \nu$. 
A similar argument as that for  \eqref{co4.4} brings us to 
$$
\vartheta(\Xi(u) +t_k(\xi^k+\nu)) = \vartheta(\Xi(u)+t_k\xi^k)+\la \mu, \nu\ra,
$$ 
which in turn yields 
\begin{equation}\label{co6.2}
\Delta_{t_k}^2\vartheta(\Xi(u),  \mu)(\xi^k) = \Delta_{t_k}^2\vartheta(\Xi(u),  \mu)(\xi^k+\nu).
\end{equation}
Observing that this equality is, indeed,  valid for any $\xi^k \in \Y$, we arrive at a similar equality for the second subderivative of $\vartheta$, namely 
$$
\d^2\vartheta(\Xi(u),  \mu)(\nabla \Xi(u)w) = \d^2\vartheta(\Xi(u),  \mu)(\nabla \Xi(u)w+\nu) = \d^2\vartheta(\Xi(u),  \mu)(\xi).
$$
Using this, we infer from \eqref{co7} and \eqref{co6.2} that 
\begin{equation*}\label{co7.1}
\lim_{k\to\infty} \Delta_{t_k}^2\vartheta(\Xi(u),  \mu)(\xi^k+\nu) =\d^2\vartheta(\Xi(u),  \mu)(\xi).
\end{equation*}
Since  $\xi^k +\nu \to \nabla \Xi(u)w + \nu = \xi$ and since $\xi\in \Z$ was taken arbitrary,  the above equality confirms that  $\vartheta$ is twice epi-differentiable at $\Xi(u)$ for $ \mu$. 
Recalling the condition $\oy \in \ri \partial g (\ou)$, we get from \eqref{rint} that $\bar\mu \in \ri \partial \vartheta(\Xi(\ou))$. 
Thus, Proposition~\ref{septheta} demonstrates that   $\vartheta $ is strictly twice epi-differentiable at $\Xi(\ou) = 0$ for $\bar\mu$. 
We know from Proposition~\ref{prop:nondeg}(b) that the dual condition in \eqref{duq} is satisfied. 
Appealing now to \cite[Proposition~7.1]{HaS23} tells us  that   $\mu \to \bar\mu$  as $(u,y)\to (\ou, \oy)$. 
So, we conclude from Proposition~\ref{chste}  that $\d^2\vartheta(\Xi(u),  \mu) \xrightarrow{e} \d^2\vartheta(\Xi(\ou), \bar\mu)$. It then follows from  \cite[Exercise~7.47(a)]{rw}   that $\d^2\vartheta(\Xi(u),  \mu)\circ \nabla \Xi(u)\xrightarrow{e} \d^2\vartheta(\Xi(\ou), \bar\mu)\circ\nabla \Xi(\ou)$ if  the condition 
\begin{equation*}\label{cq3}
0\in \inte \big(K_\vartheta(\Xi(\ou),\bar \mu)-\rge \nabla \Xi(\ou)\big)
\end{equation*}
is satisfied. We know from \cite[Proposition~2.97]{bs} that  this condition is equivalent to the dual condition in \eqref{duq} with $(u,\mu)=(\ou,\bar\mu)$, which clearly holds under the nondegeneracy condition in \eqref{nondeg}.
Finally, the sum rule for epi-convergence in \cite[Theorem~7.46(b)]{rw}  yields
\begin{equation*}
\la  \mu, \nabla^2 \Xi(u)(\cdot, \cdot)\ra +\d^2\vartheta(\Xi(u),  \mu)(\nabla\Xi(u) \cdot) \xrightarrow{e} \la \bar\mu, \nabla^2 \Xi(\ou)(\cdot, \cdot)\ra +\d^2\vartheta(\Xi(\ou), \bar\mu)(\nabla \Xi(\ou)\cdot),
\end{equation*}
which is equivalent by  \eqref{co3} to $\d^2g (u, y)$ epi-converging to $ \d^2g (\ou, \oy)$  as $(u,y)\to (\ou,\oy)$  in the set of pairs  $(u,y)\in\gph \sub g$ for which $g$ is   twice epi-differentiable. Appealing to Proposition~\ref{chste} confirms  strict twice epi-differentiability of $g $ at $\ou$ for $\oy$. 

We now proceed with verifying similar conclusion for any pair $(u,y)\in \gph \sub g$ in some neighborhood of $(\ou,\oy)$.
Observe from Lemma~\ref{claim2} that there exists $\varepsilon>0$ such that $\mu \in \ri \partial \vartheta(z)$ for all $(z,  \mu) \in \big(\gph \partial \vartheta\big)\cap \B_\varepsilon(\Xi(\ou), \bar\mu)$. 
Take $(u,y)\in \gph \sub g$ sufficiently close to $(\ou,\oy)$ so that $g$ is reliably $\C^2$-decomposable at $u$, and that $(\Xi(u), \mu) \in \big(\gph \partial \vartheta\big)\cap \B_\varepsilon(\Xi(\ou), \bar\mu)$, where $\mu$ is the unique multiplier in $M(u, y)$ -- the multiplier set defined in \eqref{laset} associated with the composite representation \eqref{co} of $g$ around $u$. 
The latter inclusion, which can be ensured via \cite[Proposition~7.1]{HaS23},   tells us that $\mu\in \ri \sub \vt(\Xi(u))$.
Employing the equivalence from \eqref{rint} again, shrinking the neighborhood around $\ou$ if necessary, we have $y \in \ri \sub g (u)$.
Repeating arguments carried out in the above proof for $(\ou, \oy)$ demonstrates that $g$  is strict twice epi-differentiable  at $u$ for $y$ and hence proves (a). 
This completes the proof.
\end{proof}

 \begin{Remark}\label{proto} Observe from the proof of Theorem~\ref{c2cone} that had we assumed the reliable $\C^2$-decomposability of $g$  only at $\ou$, we would have ensured strict twice epi-differentiability of $g$ in (a)
 and strict proto-differentiability of $\sub g$ in (b)  only at $\ou$ for $\oy$. We should also point out that the reliable $\C^2$-decomposability of $g$ in  a neighborhood of  $\ou$, used in   Theorem~\ref{c2cone}, 
 is not restrictive, since most functions,  satisfying the reliable $\C^2$-decomposability, enjoy this property at all points in their domains; see, e.g., Corollary~\ref{sdp}.
 \end{Remark}
 
 Theorem~\ref{c2cone} is a far-reaching extension of our recent results in \cite[Theorem~4.3]{HJS22} in which a similar characterization of strict twice epi-differentiability was achieved 
 for polyhedral functions. We should point out that   \cite[Theorem~3.9]{HaS22} presented a similar result for a composite function with the outer function being a polyhedral function and  the nondegeneracy condition being satisfied.
The latter result can be distilled from   Theorem~\ref{c2cone} using the observation in Example~\ref{c2rex}(b) for such  composite functions.
 
  \begin{Corollary}\label{sdp}
 Assume that   $(\overline X,\overline Y)\in \gph N_{\S_+^n}$.
 Then the following properties are equivalent:
 \begin{enumerate}[noitemsep,topsep=2pt]
\item $\dd_{\S_+^n}$ is strictly twice epi-differentiable at $X$ for $Y$ for any pair $(X,Y)\in \gph N_{\S_+^n}$ in a neighborhood of $(\overline X,\overline Y)$;  

\item $N_{\S_+^n}$ is strictly proto-differentiable at $X$ for $Y$ for any pair $(X,Y)\in \gph N_{\S_+^n}$ in a neighborhood of $(\overline X,\overline Y)$;  

\item $\overline Y \in \ri N_{\S_+^n} (\overline X)$;

\item  $\rank \overline X+\rank \overline Y=n$.
\end{enumerate} 
 \end{Corollary} 
 \begin{proof}  It is known that the cone of positive semidefinite matrics $\S_+^n$ 
 is $\C^2$-cone reducible -- and thus its indicator function is reliably $\C^2$-decomposable -- at any of its points; see  \cite[Example~3.140]{bs}.
 Employing now Theorem~\ref{c2cone} justifies the equivalences of (a)-(c). The equivalence of (c) and (d) is well-known; see \cite[page~320]{bs}.
 \end{proof}

 Several important consequences can be distilled from  Theorem~\ref{c2cone}. We begin with a useful characterization of continuous differentiability of the proximal mapping 
 of $\C^2$-decomposable functions.

\begin{Theorem}\label{diffprox}
Assume that $g: \Y\to \oR$, $(\ou,\oy) \in \gph  \partial g$, and that there is a neighborhood $U$ of $\ou$ such that  $g$ is reliably $\C^2$-decomposable at any $u\in U$. 
If $g$ is prox-bounded, then  the following properties are equivalent:
\begin{enumerate}[noitemsep,topsep=2pt]
\item $\oy \in \ri \partial g (\ou)$;
\item [ \rm {(b)}]  there exists a positive constant   $r$ such that for any $\gamma\in (0,1/r)$, the proximal mapping $\prox_{\gamma g}$ is ${\cal C}^1$ around $\ou+\gamma\oy$;
\item [ \rm {(c)}]  there exists a positive constant   $r$ such that for any $\gamma\in (0,1/r)$, the envelope function   $e_{\gamma} g$ is ${\cal C}^2$ around $\ou+\gamma\oy$.
\end{enumerate}
In addition, if $g$ is convex, the constant $r$ in {\rm(}b{\rm)} and {\rm(}c{\rm)} can be taken as $0$ with convention $1/0=\infty$. 
\end{Theorem}
\begin{proof} According to Proposition~\ref{prop:nondeg}(a), $g$ is both prox-regular and subdifferentially continuous at $\ou$ for $\oy$. Appealing to 
Theorem~\ref{c2cone} tells us that (a)  is equivalent to strict proto-differentiability of $\sub g$ at $u$ for $y$ for any $(u,y)\in \gph \sub g$ sufficiently close to $(\ou,\oy)$.
It follows   from  Theorem~\ref{chpr} that (a)-(c) are equivalent. 
\end{proof}

For the indicator function of a ${\cal C}^2$-cone reducible set, the above result is simplified as recorded below.
\begin{Corollary}\label{c2red}
Assume that $C\subset \Y$ is a nonempty closed convex subset, $\ou\in C$, and that $C$ is ${\cal C}^2$-cone reducible set at any $u$ in a neighborhood of $\ou$.
Then  the following properties are equivalent:
\begin{enumerate}[noitemsep,topsep=2pt]
\item $\oy \in \ri   N_C (\ou)$;
\item [ \rm {(b)}]   for any $\gamma>0$, the proximal mapping $\prox_{\gamma g}$ is ${\cal C}^1$ around $\ou+\gamma\oy$;
\item [ \rm {(c)}]    for any $\gamma >0$, the envelope function   $e_{\gamma} g$ is ${\cal C}^2$ around $\ou+\gamma\oy$.
\end{enumerate}
In particular, the properties in {\rm(}a{\rm)}-{\rm(}c{\rm)} are equivalent when $C=\S^n_+$.
\end{Corollary}
\begin{proof}
The claimed equivalence of (a)-(c) results from Theorems~\ref{diffprox} and \ref{chpr}. The last claim about $\S^n_+$ falls out of Corollary~\ref{sdp}. 
\end{proof}

Note that Shapiro in \cite[Proposition~3.1]{sh16} characterized the differentiability of projection onto ${\cal C}^2$-cone reducible convex sets 
 using a similar relative interior condition in Corollary~\ref{c2red}(a). Shapiro's result was generalized in \cite[Lemma~5.3.32]{mi} for ${\cal C}^2$-decomposable convex functions.
Corollary~\ref{c2red} and  Theorem~\ref{diffprox}, respectively, improve both latter results by showing that indeed differentiability can be strengthened to continuous differentiability.
We should, however, point out that while the authors in \cite{sh16} and \cite{mi} assume ${\cal C}^2$-cone reducibility and  reliable ${\cal C}^2$-decomposability  at the point under consideration,
Corollary~\ref{c2red} and  Theorem~\ref{diffprox} requires that these properties   be satisfied in a neighborhood of such a point.  Also, it was shown in \cite[Theorem~28]{dhm}  that the proximal mapping of any prox-bounded and prox-regular ${\cal C}^2$-partly smooth function (cf. \cite[Definition~14]{dhm})  satisfying the relative interior condition in Theorem~\ref{diffprox}(c) enjoys local ${\cal C}^1$-smoothness similar to that in Theorem~\ref{diffprox}(b).  It was shown by Shapiro in \cite{sh03} that 
   reliably ${\cal C}^2$-decomposable functions are ${\cal C}^2$-partly smooth. This suggests that the implication (a)$\implies$(b)
   in  Theorem~\ref{diffprox} can be extracted from  \cite[Theorem~28]{dhm}. Note that it is possible to show that ${\cal C}^2$-partly smooth functions are 
   strictly twice epi-differentiable. This makes it possible to completely recover   \cite[Theorem~28]{dhm} using our approach and even shows that the relative interior condition 
   indeed characterizes the continuous differentiability of proximal mappings for this class of functions. We will leave this issue for our forthcoming paper \cite{HaS24}.   

Next, we provide a characterization of the existence of a continuously differentiable single-valued graphical localization of the solution mapping to 
the generalized equation 
\begin{equation}\label{deeq}
0\in \psi(u)+ \sub g(u),
\end{equation} 
where $\psi:\Y\to \Y$ is a continuously differentiable function and $g:\Y\to \oR$ is a proper function. When $g$ enjoys the reliable  
$\C^2$-decomposability, we can apply Theorem~\ref{c1fors} to obtain the result below. In what follows, we call a solution 
$\ou$ to the generalized equation in \eqref{deeq} {\em nondegenerate} if $-\psi(\ou)\in \ri \sub g(\ou)$.

\begin{Corollary} Assume that $\ou$ is a  solution to the generalized equation in \eqref{deeq}, that 
 there is a neighborhood $U$ of $\ou$ such that  $g$ in \eqref{deeq}  is reliably $\C^2$-decomposable at any $u\in U$. Then, the following properties are equivalent:
\begin{enumerate}[noitemsep,topsep=2pt]
\item the mapping $\psi+\sub g$  is metrically regular at $\ou$ for $0$ and  $\ou$ is a nondegenerate solution to \eqref{deeq};
\item the solution mapping $S$, defined by 
$$
S(y) :=  \big\{ u\in \Y\, \big|\, y\in \psi(u) +\partial g(u)\big\}, \quad y\in \Y,
$$
 has a Lipschitz continuous single-valued localization $s$ around $0\in \Y$ for $\ou$, which  is ${\cal C}^1$ around $0$. 
\end{enumerate} 
In particular, the properties in {\rm(}a{\rm)} and {\rm(}b{\rm)} are equivalent when $g=\dd_{\S^n_+}$.
\end{Corollary}
\begin{proof}
The claimed equivalence results from Theorems~\ref{c2cone} and \ref{c1fors}. 
The last claim about $\S^n_+$ falls out of Corollary~\ref{sdp}.
\end{proof}

We close this section with an application of our main result in studying stability properties of the KKT system of the  composite optimization problem 
\begin{equation}\label{cp}
\mini \varphi(x) + (g\circ\Phi)(x)\quad \mbox{subject to}\quad x\in \X, 
\end{equation}
where $\varphi: \X \to \R$ and $\Phi: \X \to \Y$ are $\C^2$-smooth mappings and $g:\Y \to \oR$ is a convex function.
We are going to impose the reliable  ${\cal C}^2$-decomposability assumption on $g$ to make use of our developments of strict proto-differentiability of 
subgradient mappings in this section. With that in mind, the composite program in \eqref{cp} encompasses important classes of constrained and composite optimization problems including   second-order cone programming problems, when  $g = \delta_{\cal Q}$, where ${\cal Q}$ stands for  the second-order cone, and semidefinite programming  problems, when  $g = \delta_{\S^m_+}$.
Given $(x,  y)\in \X\times \Y$, define the Lagrangian of \eqref{cp} by $L(x,  y):= \varphi(x) +\la  y, \Phi(x)\ra$. 
The KKT system associated with the composite problem \eqref{cp} is given by
\begin{equation}\label{kkt}
0=\nabla_xL(x,  y), \quad  y \in \partial g(\Phi(x)).
\end{equation}
 Define the mapping $\Psi:\X\times\Y \tto\X\times \Y$ by
\begin{equation}\label{mr10}
\Psi(x,  y):= \begin{bmatrix}
\nabla_xL(x,  y)\\-\Phi(x)
\end{bmatrix} +\begin{bmatrix}
0\\ \partial g^*( y)
\end{bmatrix}.
\end{equation}
It can be easily seen that a pair   $(\ox, \oy)\in \X\times \Y$ is  a solution to the KKT system in  \eqref{kkt} if and only if $(0, 0)\in \Psi(\ox, \oy)$. 
Define also the  solution mapping $S:\X\times\Y \tto\X\times \Y$ to the canonical perturbed of the KKT system in \eqref{kkt} by
\begin{equation*}
S_{KKT}(p, q):= \Psi^{-1}(p, q)= \big\{(x,  y) \in \X\times \Y\, \big|\, (p,q) \in \Psi(x,  y)\big\},\quad\textrm{ for }\; (p,q)\in \X\times \Y.
\end{equation*}
The next result presents a characterization of   strong metric regularity of the   solution mapping $S_{KKT}$ under a relative interior condition imposed for the Lagrange multiplier under consideration.  

\begin{Theorem}\label{mrkkt}
Assume that  $(\ox, \oy)$ is a solution to the  KKT system in  \eqref{kkt} and   $\oy\in \ri \partial g(\ou)$ with $\ou :=\Phi(\ox)$,   that the convex function $g$  in \eqref{cp} is 
 reliably $\C^2$-decomposable in a neighborhood of  $\ou$.
Then the following properties are equivalent:
\begin{itemize}[noitemsep,topsep=2pt]
\item [ \rm {(a)}] the mapping $\Psi$ is strongly metrically regular at $(\ox, \oy)$ for $(0,0)$;
\item [ \rm {(b)}] the mapping $\Psi$ is  metrically regular at $(\ox, \oy)$ for $(0,0)$;
\item [ \rm {(c)}] the mapping $\Psi$ is strongly metrically subregular at $(\ox, \oy)$ for $(0,0)$;
\item [ \rm {(d)}] the solution mapping $S_{KKT}$ has a Lipschitz continuous single-valued localization around $(0, 0)$ for $(\ox, \oy)$, which is  ${\cal C}^1$ in a neighborhood of  $(0, 0)$;
 \item [ \rm {(e)}] the implication 
\begin{equation*}\label{mrchar}
 \begin{cases}
 \nabla^2_{xx}L(\ox,\oy)w-\nabla\Phi(\ox)^*w'=0,\\
 \nabla \Phi(\ox) w \in K_g(\ou,\oy),\\
 w'+ \nabla^2\la \bar\mu, \Xi\ra(\ou)(\nabla \Phi(\ox) w)\in K_g(\ou,\oy)^\perp
 \end{cases}
 \implies (w,w')=(0,0)
 \end{equation*}
 holds, where $\bar\mu\in \Z$ is the unique element  in the Lagrange multiplier set $M(\ou, \oy)$ from \eqref{laset}.
 \end{itemize}
\end{Theorem}

\begin{proof}
Setting 
\begin{equation*}\label{mr10.1}
\psi(x,  y): = \begin{bmatrix}
\nabla_xL(x,  y)\\
-\Phi(x)
\end{bmatrix} \quad\textrm{ and }\quad f(x,  y) := g^*( y),\quad (x, y)\in \X\times \Y,
\end{equation*}
and observing via \cite[Proposition~10.5]{rw} that 
\begin{equation}\label{mr10.2}
\partial f (x,  y) = \{0\}\times \partial g^{*}( y),
\end{equation} 
we can equivalently write the KKT system in \eqref{kkt} as the generalized equation 
$$
0\in \psi(x,  y)+\sub f(x,  y).
$$
Since $(\ox,\oy)$ is a solution  to the KKT system in \eqref{kkt}, it is a solution to the above generalized equation. 
It follows from Theorem~\ref{c2cone} and $\oy \in \ri \partial g(\ou)$ that there is $\ve>0$ such that  $\sub g$ is strictly proto-differentiable at $u$ for $y$ for any $(u,y)\in (\gph \sub g)\cap \B_\ve(\ou,\oy)$. 
Since $\sub g^*=(\sub g)^{-1}$ due to the convexity of $g$, we deduce that $\sub g^*$ is strictly proto-differentiable at $y$ for $u$ for any $(y,u)\in (\gph \sub g^*)\cap \B_\ve(\oy,\ou)$. 
By \eqref{mr10.2}, we have 
$$
(x,y,v,u)\in \gph \partial f \iff (x,v, y,u)\in  \X\times \{0\} \times \gph \partial g^*.
$$
Pick $(x,y,v,u)\in \gph \partial f$ so that $(y,u)\in (\gph \sub g^*)\cap \B_\ve(\oy,\ou)$.  Using the definition of the regular tangent cone, we obtain 
$$
(\xi_1,\xi_2,\eta_1,\eta_2)\in \rt_{\gph \sub f}(x,y,v,u)\iff (\xi_1,\eta_1,\xi_2,\eta_2)\in \X\times \{0\} \times \rt_{\gph \sub g^*}(y,u).
$$
Similarly, it is not hard to see via the definition of the paratingent cone that 
$$
(\xi_1,\xi_2,\eta_1,\eta_2)\in \widetilde T_{\gph \sub f}(x,y,w,u)\iff (\xi_1,\eta_1,\xi_2,\eta_2)\in \X\times \{0\} \times \widetilde T_{\gph \sub g^*}(y,u).
$$
Since $\sub g^*$ is strictly proto-differentiable at $y$ for $u$, we have $ \rt_{\gph \sub g^*}(y,u)=\widetilde T_{\gph \sub g^*}(y,u)$. Thus, we arrive at 
$\rt_{\gph \sub f}(x,y,v,u)=\widetilde T_{\gph \sub f}(x,y,v,u)$, which is equivalent to saying that $\sub f$ is strictly proto-differentiable at $(x,y)$ for $(v,u)$ whenever $(y,u)\in (\gph \sub g^*)\cap \B_\ve(\oy,\ou)$.
Observe also that $f$   is clearly prox-regular and subdifferentially continuous at $(\ox, \oy)$ for $(0, \ou)$, since $g$ is convex. 
Appealing now to Theorem~\ref{c1fors} demonstrates that (a), (b), and (d) are equivalent. On the other hand, it is easy to see that $\nabla \psi(\ox,\oy)=\nabla \psi(\ox,\oy)^*$. Thus, it results from 
Corollary~\ref{thm:MR4} that (b) and (c) are also equivalent. 

Turning to (e), recall that $\sub f$ is strictly proto-differentiable at $(\ox, \oy)$ for $(0, \ou)$. By Theorem~\ref{thm:gdcod} and  \eqref{mr10.2}, we conclude for any $(w, w')\in \X\times \Y$ that  
\begin{eqnarray*}
D^*(\partial f)\big((\ox, \oy), (0, \ou)\big)(w, w')=D(\partial f)\big((\ox, \oy), (0, \ou)\big)(w, w') = \{0\}\times D(\partial g^*)(  \oy,\ou)(w').
\end{eqnarray*}
Using this, the sum rule for coderivatives from \cite[Exercise~10.43(b)]{rw}, we get for any $(w, w')\in \X\times \Y$ that
\begin{align*}
D^*(\psi +\partial f)\big((\ox,&\, \oy), (0, 0)\big)(w, w')\\ 
&= \nabla \psi(\ox, \oy)^*(w, w') + D^*(\partial f)\big((\ox, \oy), (0, \Phi(\ox))\big)(w, w')\\
&= \nabla \psi(\ox, \oy)^*(w, w') + D(\partial f)\big((\ox, \oy), (0, \Phi(\ox))\big)(w, w')\\
&= \big(\nabla^2_{xx}L(\ox, \oy)w - \nabla\Phi(\ox)^*w', \nabla \Phi(\ox) w\big)+\{0\}\times D(\partial g^*)(  \oy,\ou)(w')\\
&= \big(\nabla^2_{xx}L(\ox, \oy)w - \nabla\Phi(\ox)^*w', \nabla \Phi(\ox) w\big)+\{0\}\times D(\partial g)(\ou, \oy)^{-1}(w'),
\end{align*}
where the last equality comes from $\sub g^*=(\sub g)^{-1}$. By  \cite[Theorem~3.3(ii)]{Mor18},  the mapping $\Psi=\psi +\partial f$  is metrically regular at $(\ox, \oy)$ for $(0, 0)$ if and only if 
the implication 
$$
(0,0)\in D^*\Psi \big((\ox, \oy), (0, 0)\big)(w, w')=D^*(\psi +\partial f)\big((\ox, \oy), (0, 0)\big)(w, w')\implies w=0,\; w'=0
$$
is satisfied. 
To show that this implication is the same as the one in (e), assume $(0,0)\in D^*\Psi \big((\ox, \oy), (0, 0)\big)(w, w')$. By the calculation above, we   get 
$\nabla^2_{xx}L(\ox, \oy)w - \nabla\Phi(\ox)^*w'=0$ and $- \nabla \Phi(\ox) w\in D(\partial g)(\ou, \oy)^{-1}(w')$. The latter amounts to the inclusion $w'\in D(\partial g)(\ou, \oy)(- \nabla \Phi(\ox) w)$.
Employing \cite[Theorem~6.2(b)]{HaS23} tells us that 
$$
 D(\partial g)(\ou, \oy)(- \nabla \Phi(\ox) w)= -\nabla^2\la \bar\mu, \Xi\ra(\ou)(\nabla \Phi(\ox) w)+ N_{K_g(\ou,\oy)}(-\nabla \Phi(\ox) w).
$$
Since $K_g(\ou,\oy)$ is a linear subspace of $\Y$ due to $\oy \in \ri \partial g(\ou)$, we have 
$$
N_{K_g(\ou,\oy)}(-\nabla \Phi(\ox) w)=\begin{cases}
K_g(\ou,\oy)^\perp& \mbox{if}\; \; \nabla \Phi(\ox) w\in K_g(\ou,\oy),\\
\emptyset& \mbox{otherwise}.
\end{cases} 
$$
Combining these shows that (e) is equivalent to (b) and hence completes the proof.
 \end{proof}

For   classical nonlinear programming problems (NLPs), it is well-known that metric regularity and strong metric regularity of KKT systems are equivalent; see \cite[Theorem~4I.2]{DoR14} and \cite[Section~7.5]{kk}.
By using a new approach, Theorem~\ref{mrkkt} extends this result for the composite problem \eqref{cp} under a relative interior condition. 
This extra condition allows us to demonstrate further that the Lipschitz continuous single-valued localization of the solution mapping to the KKT system of \eqref{cp}
is continuously differentiable. This can be viewed as an extension of Fiacco and McCormick's result in \cite{fm} for NLPs, which was achieved under the 
classical second-order sufficient condition, strict complementarity condition, and linear independence constraint qualification.

\section*{Declarations}

{\bf   Conflicts of interests.}  The authors have no competing interests to declare that are relevant to the content of
this article.

\end{document}